\documentclass[11pt, a4paper, reqno]{amsart}
\usepackage{amsmath}
\usepackage{amsfonts}
\usepackage{amssymb}
\usepackage{amsthm}
\usepackage{url}
\usepackage{color}
\usepackage{array}
\usepackage{hyperref}

\newcommand{\R}{\mathbb{R}}

\newcommand{\N}{\mathbb{N}}

\DeclareMathOperator*{\diam}{diam}

\DeclareMathOperator*{\dist}{dist}

\DeclareMathOperator*{\Lip}{Lip}

\DeclareMathOperator{\Mod}{Mod}
\DeclareMathOperator{\capa}{Cap}
\DeclareMathOperator{\rcapa}{cap}

\DeclareMathOperator*{\rad}{rad}

\newcommand{\rcapap}{\text{cap}^\pip_p}

\def\vint_#1{\mathchoice%
	{\mathop{\kern 0.2em\vrule width 0.6em height 0.69678ex depth -0.58065ex
			\kern -0.8em \intop}\nolimits_{\kern -0.4em#1}}%
	{\mathop{\kern 0.1em\vrule width 0.5em height 0.69678ex depth -0.60387ex
			\kern -0.6em \intop}\nolimits_{#1}}%
	{\mathop{\kern 0.1em\vrule width 0.5em height 0.69678ex depth -0.60387ex
			\kern -0.6em \intop}\nolimits_{#1}}%
	{\mathop{\kern 0.1em\vrule width 0.5em height 0.69678ex depth -0.60387ex
			\kern -0.6em \intop}\nolimits_{#1}}}

\newcommand{\Om}{\Omega}

\newcommand{\eps}{\varepsilon}
\newcommand{\pip}{\varphi}

\newcommand{\ch}{\text{\raise 1.3pt \hbox{$\chi$}\kern-0.2pt}}

\newcommand{\Qpip}{Q_{\mu_\pip}^-}

\theoremstyle{plain}
\newtheorem{theorem}[equation]{Theorem}
\newtheorem{lemma}[equation]{Lemma}
\newtheorem{corollary}[equation]{Corollary}
\newtheorem{proposition}[equation]{Proposition}

\numberwithin{equation}{section}

\theoremstyle{definition}
\newtheorem{definition}[equation]{Definition}
\newtheorem{deff}[equation]{Definition}
\newtheorem{example}[equation]{Example}

\theoremstyle{remark}
\newtheorem{remark}[equation]{Remark}

\begin{document}
	
	\title[Dirichlet problem on unbounded domains]{Solving a Dirichlet problem on unbounded domains via a conformal transformation}
	
	\author[Gibara, Korte, Shanmugalingam]{Ryan Gibara, Riikka Korte, Nageswari Shanmugalingam}
	
	\date{\today}
	
	\keywords{Uniform domain, metric measure spaces, conformal transformation, Dirichlet problem, Poincar\'e inequalities, doubling
		measures.}
	
	\subjclass{Primary: 31E05; Secondary: 30L99, 49Q05, 26A45.}
	
	\begin{abstract}
		In this paper, we solve the $p$-Dirichlet problem for Besov boundary data on unbounded uniform domains with
		bounded boundaries when the domain is equipped with a doubling measure satisfying a Poincar\'{e} inequality. 
		This is accomplished 
		by studying a class of transformations that have been recently shown to render the domain bounded while maintaining uniformity. 
		These transformations conformally deform the metric and measure in a way that depends on the distance to the boundary 
		of the domain and, for the measure, a parameter $p$. We show that the transformed measure is doubling and the 
		transformed domain supports a Poincar\'{e} inequality. This allows us to transfer known results for bounded 
		uniform domains to unbounded ones, including trace results and Adams-type inequalities, culminating in a 
		solution to the Dirichlet problem for boundary data in a Besov class. 
	\end{abstract}
	
	\maketitle
	
	\tableofcontents

	\section{Introduction}

	In studying Dirichlet and Neumann boundary-value problems on domains in metric measure spaces of bounded geometry,
	existence of the solution via the direct method of the calculus of variations requires that we are able to bound the $L^p$-norm of
	a Sobolev function on the domain (with zero boundary values) by the Sobolev energy norm of the function, thus
	ensuring the boundedness in the Sobolev norm of an energy minimizing sequence of Sobolev functions in the domain.
	When the domain is a bounded uniform domain, this is always possible thanks to the Poincar\'e inequality, for we can
	then envelop the domain in a sufficiently large ball. 
	Bounded uniform domains play a central role in potential theory as many of the classical results about Dirichlet problems on 
	smooth Euclidean domains hold for such domains in metric measure spaces. In particular, they 
	are extension domains for several function spaces~\cite{BS,VG} and traces, to the boundary, of Sobolev-class
	functions on the domain belong to certain Besov classes~\cite{Maly} of functions on the boundary. 
	However, when the domain, albeit uniform, is not
	bounded, these properties might not hold. 
	Therefore, it is beneficial to have a transformation of the domain into a bounded
	uniform domain.

	One such transformation is sphericalization
	as defined in the work of~\cite{BHK}, which transforms an unbounded metric space 
	$X$ into a bounded metric space whose completion is topologically its one-point compactification. 
	A generalization of the familiar stereographic projection, sphericalization has been used in the study of 
	quasiconformal geometry, including the study of quasi-M\"obius maps and Gromov hyperbolic 
	spaces, see for example~\cite{BK,BHX,HSX,L,ZLL}. 
	Sphericalization is known to preserve many desirable properties of the metric space $X$, see for 
	example~\cite{DL1,DL2}. In particular, if $X$ is a uniform domain (or a uniform space, in the 
	language of~\cite{BHK}), then its sphericalization is also a uniform domain, see~\cite{BHX}.
	
	On the other hand, sphericalization distorts the metric of $X$ everywhere, including near its boundary if $X$ is not 
	complete. This poses a problem if one is interested in gaining information about, or preserving the geometry of, the boundary
	of the original unbounded domain $X$ when $\partial X=\overline{X}\setminus X$  
	itself is bounded, for example as in~\cite{CKKSS}. 
	This issue was addressed in~\cite{GS}, where a class of transformations was identified such that 
	unbounded uniform domains are transformed to bounded uniform domains in such a way that the 
	inner length metric is not perturbed, locally, near the boundary. The purpose of the present paper 
	is to explore potential theory on these transformed domains, with the view of 
	applying this in ongoing work on boundary-value problems on unbounded domains. 
	
	\medskip
	
	\noindent{\bf The setting:} 
	We consider a locally compact, non-complete metric space $(\Om,d)$, equipped with a doubling 
	measure $\mu$, and supporting a Poincar\'{e} inequality, at least for balls with radius at most some fixed 
	constant times the distance from its center to the boundary. We assume that $\Om$ is unbounded and is
	uniform in its completion $\overline{\Om}$ with bounded boundary $\partial\Om:=\overline{\Om}\setminus\Om$, and 
	we fix a monotone decreasing continuous function $\pip:(0,\infty)\to(0,\infty)$ that will act as a 
	dampening function, see Definition~\ref{def:phi} for the specific assumptions on $\pip$. As $\Om$ is a 
	uniform domain, it is rectifiably connected, that is, pairs of points in $\Om$ can be connected by curves in 
	$\Om$ of finite length. As such, we may use $\pip$ to construct a new metric $d_\pip$ on $\Om$ by setting 
	\[
	d_\pip(x,y):=\inf_\gamma\ \int_\gamma\pip\circ d_\Om\, ds,
	\]
	with the infimum ranging over all rectifiable curves in $\Om$ with end points $x,y\in\Om$. 
	Here, $\int_\gamma h\, ds:=\int_\gamma h(\gamma(\cdot))\,ds$ is the path integral with respect 
	to the arc-length parametrization of the
	rectifiable curve $\gamma$, see for example~\cite[Chapter 5]{HKST}, and $d_\Om$ is defined by 
	$d_\Om(x)=\dist(x,\partial\Om)$. We 
	fix $1\le p<\infty$ and 
	also construct a new measure $\mu_\pip$ supported on $\Om$, absolutely continuous with $\mu$, with 
	Radon-Nikodym derivative $\pip(d_\Omega(x))^p$.
	
	\medskip
	
	In~\cite{GS}, it was shown that $\Om_\pip:=\overline{\Om\cup\partial\Om}^\pip\setminus\partial\Om$,
	where $\overline{A}^\pip$ is the completion of $A\subset\overline{\Om}$ with respect to the metric $d_\pip$, 
	differs from $\Om$ by one point, which we denote by $\infty$. The transformed space $(\Om_\pip,d_\pip)$ was 
	shown to be uniform in its completion and to have boundary $\partial{\Om_\pip}=\partial\Om$. Moreover, 
	$d_\pip$ and $d$ are uniformly locally bi-Lipschitz near $\partial\Om$. The present paper begins by showing the following 
	(see Theorems~\ref{thm:double} and~\ref{thm:PI-no-infty}).
	
	\begin{theorem}
		The metric measure space $(\Om_\pip,d_\pip,\mu_\pip)$ is doubling and supports a $p$-Poincar\'{e} inequality.
	\end{theorem}
	
	In the process of proving the above theorem, we verify that $\Om_\pip\setminus\{\infty\}$ is also a uniform
	domain (see Theorem~\ref{thm:uniform}), supplementing the results from~\cite{GS}.
	
	With these tools in hand, in Sections~\ref{Sect:Dirichlet} and~\ref{Sec:eight}
	we proceed to study potential theory on the domain $\Om_\pip$. We show that a function is 
	$p$-harmonic on $(\Om, d,\mu)$ if and only if it is $p$-harmonic on 
	$(\Om_\pip\setminus\{\infty\},d_\pip,\mu_\pip)$. Furthermore, when $p$ is 
	sufficiently large and $\pip(t)=\min\{1,t^{-\beta}\}$ for $t>0$ and some fixed  large enough
	$\beta>1$, $p$-harmonic functions on $(\Om_\pip\setminus\{\infty\},d_\pip,\mu_\pip)$ 
	can be extended to become $p$-harmonic on all of $(\Om_\pip,d_\pip,\mu_\pip)$. We show in
	Proposition~\ref{thm:bdry-mod} that when the index $p$ is small, the $p$-capacity of $\{\infty\}$ is zero;
	but if $p$ is sufficiently large, then the $p$-capacity of $\{\infty\}$ is positive. We also show that when the
	measure on $\partial\Om$ satisfies a codimensionality condition with respect to $\mu$, the trace class of the Dirichlet-Sobolev
	space $D^{1,p}(\Om,d,\mu)$ is a Besov space of functions on $\partial\Om$, see Proposition~\ref{prop:Dp-trace},
	and an Adams-type inequality holds for the measure on $\partial\Om$ and functions in $D^{1,p}(\Om,d,\mu)$,
	see Theorem~\ref{thm:Adams}. Interestingly, it turns out that under this codimensionality condition for the boundary,
	each (relative) ball in $\partial\Om$ has positive $p$-capacity in $\overline{\Om_\pip}^\pip$, see
	Proposition~\ref{prop:bdryCap}.
	
	Using the potential theory developed in Sections~\ref{Sect:Dirichlet} and~\ref{Sec:eight}, we 
	obtain the following culminating theorem regarding the Dirichlet problem in Section~\ref{Sec:nine}. In what follows, the boundary data $f$ is taken to be in the Besov space $B^{1-\theta/p}_{p,p}$ on the boundary $\partial\Om$ with respect to a codimensional measure $\nu$, see Definition~\ref{Besov} for the definition of the Besov space and Proposition~\ref{prop:Besov-trace} regarding the trace operator $T$ acting on the Dirichlet-Sobolev space $D^{1,p}(\Om,\mu)$; for the description of
	$D^{1,p}$, we refer the reader to the paragraph before Definition~\ref{def:poin} below.
	
	\begin{theorem}\label{thm:one-point-two}
		Let $1<p<\infty$ and $\nu$ be a measure on $\partial\Om$ that is $\theta$-codimensional with respect to the measure $\mu$ on $\Om$
		with $0<\theta<p$, as described at the beginning of Section~\ref{Sec:eight}.
		Let $f\in B^{1-\theta/p}_{p,p}(\partial\Om,\nu)$. Then there is a
		function $u\in D^{1,p}(\Om,\mu)$ such that
		\begin{itemize}
			\item $u$ is $p$-harmonic in $(\Om,d,\mu)$,
			\item $Tu=f$ on $\partial\Om$ $\nu$-a.e..
		\end{itemize}
		If $\Om$ is $p$-parabolic, then the solution $u$ is unique. If $\Om$ is $p$-hyperbolic, then 
		for each solution $u$ of the problem we have $\lim_{\Om\ni y\to\infty}u(y)$ exists; this limit uniquely determines
		the solution.
	\end{theorem}
	
	\medskip
	
	{\bf Acknowledgments.} The research of N.S. is partially supported by NSF grant~\#DMS-2054960.
	Part of the research in this paper was conducted during the research stay of N.S. at the 
	Mathematical Sciences Research Institute (MSRI, Berkeley, CA) as part of the program
	\emph{Analysis and Geometry in Random Spaces} which is 
	supported by the National Science Foundation (NSF) under Grant No.~1440140, during Spring 2022. She 
	thanks MSRI for its kind hospitality. Part of the research was done while R.K. visited  University of Cincinnati. 
	She wishes to thank University of Cincinnati for its kind hospitality. 
	The authors thank the two referees for a careful reading of the manuscript and for comments that helped
	improve the exposition of the paper.
	The authors also thank 
	A.~Tyulenev for suggesting a correction to an earlier version of Lemma~\ref{lem:maximal}.

	\section{Construction of the transformation of metric and measure} 
	
	Let $(\Om, d)$ be an unbounded, locally compact, non-complete metric space such that $\Om$ is a \emph{uniform 
		domain} in its completion $\overline{\Om}$. A uniform domain is one for which there exists a constant 
	$C_U\ge 1$ satisfying the following property:
	for each $x,y\in\Om$ with $x\ne y$, there is a \emph{uniform curve} with end points $x$ and $y$, that is, a curve $\gamma$ such that 
	\begin{itemize}
		\item the length (with respect to the metric $d$) of the curve $\gamma$ satisfies $\ell_d(\gamma)\le C_U\, d(x,y)$,
		\item for each $z$ in the trajectory of $\gamma$, we have
		\[
		\min\{\ell_d(\gamma[x,z]),\ell_d(\gamma[z,y])\}\le C_U\, d_\Omega(z).
		\]
	\end{itemize}
	Here, for two points $w_1,w_2$ in the trajectory of $\gamma$, we represent each segment of $\gamma$
	with end points $w_1,w_2$ by $\gamma[w_1,w_2]$. Moreover,
	$d_\Omega(x):=\dist(x,\partial\Omega)$ for $x\in\overline{\Omega}$, and $\partial\Om=\overline\Omega\setminus\Omega$.
	
	Throughout this paper, we set $n_0$ to be the smallest integer such that 
	\[
	2^{n_0-1}\le C_U<2^{n_0}.
	\] 
	We will also assume that $\partial\Om$ is bounded, that $\mu$ is a doubling Radon measure supported on
	$\Om$ with doubling constant $C_\mu\geq{1}$, and that $(\Om,d,\mu)$ supports a sub-Whitney $p$-Poincar\'{e} inequality for some fixed $1\leq p <\infty$, see Section~\ref{sec:background} for the definitions. 
	
	The notation $A\lesssim B$ will be used to mean that there exists a constant $C>0$, depending only on structural data, such that $A\leq CB$; furthermore, the notation $A\approx B$ means that $A\lesssim B$ and $A\gtrsim B$. 
	
	\begin{deff}\label{def:phi}
		In this paper, we fix a monotone decreasing continuous function $\pip:(0,\infty)\to(0,1]$ such that the following hold:
		\begin{enumerate}
			\item $\pip(t)=1$ when $0<t\le 1$.
			\item We have
			\begin{equation}\label{finite}
				\int_0^\infty\! \pip(t)\, dt<\infty.
			\end{equation}
			\item There is a constant $C_\pip\ge 1$ for which we have $\pip(t)\le C_\pip\,\pip(2t)$ for all $t>0$ (that is, $\pip$ satisfies a reverse doubling condition).
			\item There is some $\tau>2$ such that $\pip(t)\ge \tau \pip(2t)$ (thus requiring $C_\pip>2$, as well).
			\item For all positive integers $m$, we have 
			\begin{equation}\label{additionalpip}
				2^{m}\pip(2^m)\leq \sum_{n=m}^\infty 2^n\pip(2^n) \lesssim 2^{m}\pip(2^m)
			\end{equation}
			(and, indeed, this condition follows from Condition~(4) above, but we list it here for it is used extensively in this paper).
			\item For all positive integer $m$,
			\begin{equation}\label{additionalmupip}
				\pip(2^m)^p\mu(\Om_m)\le \sum_{n=m}^\infty \pip(2^n)^p\mu(\Om_n)  \lesssim  \pip(2^m)^p\mu(\Om_m),
			\end{equation}
			where, for $n>0$,
			\[
			\Om_n:=\{x\in\Om\, :\, 2^{n-1}<d_\Om(x)\le 2^n\}.
			\]
		\end{enumerate}
	\end{deff}
	
	The last condition is needed in order to know that $\mu_\pip$ is finite and doubling,
	see below for the definition of $\mu_\pip$.
	Examples of functions $\pip$ satisfying all the necessary conditions include $\pip(t)=\min\{1,t^{-\beta}\}$  or
	$\pip(t)=\min\{1,t^{-\beta}\, \log(e-1+t)\}$ for some sufficiently large fixed $\beta>1$ depending on $p$ and the doubling property of $\mu$ (see the condition on $\beta$ in Lemma~\ref{lem:lower-mass}).
	\\
	
	\noindent {\bf Construction:}
	For $1\leq p<\infty$, we wish to transform the geometry of $\Om$ by weighting 
	both the metric and the measure on $\Om$ using $\pip$ in the following way.
	
	We transform the metric $d$ on $\Omega$ into $d_\pip$ by
	setting 
	\[
	d_\pip(x,y):=\inf_\gamma \ell_\pip(\gamma):= \inf_\gamma\int_\gamma \pip(d_\Omega(\gamma(t)))\, dt
	\]
	with the infimum ranging over all rectifiable curves $\gamma$ in $\Omega$ with end points $x$ and $y$ in $\Om$. 
	Note that $\Om$ is rectifiably connected as it is uniform. The notation $B_{d_\pip}$ and $B_d$ will be used for balls 
	taken with respect to the metric $d_\pip$ and $d$, respectively. All balls will be assumed to come with a prescribed center and radius.  
	
	The measure $\mu$ on $\Om$ is transformed into the measure $\mu_\pip$, absolutely continuous with respect to $\mu$, with 
	\[
	d\mu_\pip(x)=\pip(d_\Omega(x))^p\, d\mu.
	\]
	Note that $\mu_\pip$ depends not only on $\pip$ but also on the choice of $p$; however, $1\leq p<\infty$ is fixed and we 
	suppress the dependence on $p$ in the notation. We credit~\cite{BBL} for the idea of 
	considering transformations of measures that are allowed to depend on $p$. 
	
	Now we have two identities for $\Om$; namely, $(\Om,d,\mu)$ and $(\Om,d_\pip,\mu_\pip)$. 
	Consider the set $\Om_\pip:=\overline{\Om\cup\partial\Om}^\pip\setminus\partial\Om$, where the completion 
	is taken with respect to $d_\pip$ and $\partial\Om:=\overline{\Om}\setminus{\Om}$. In~\cite{GS}, it was shown 
	that there is only one point in $\Om_\pip\setminus\Om$, which we denote by $\infty$. Moreover, 
	$\partial\Om_\pip=\partial{\Om}$ and $(\Om_\pip,d_\pip)$ is uniform in its completion. In~\cite{GS}, the uniform domain properties of
	the transformed domain $\Om_\pip$ were studied; many of the tools developed there will be used in the present paper.
	The goal of this paper is to investigate properties related to $\mu_\pip$, and apply these properties
	to the study of potential theory and the Dirichlet problem on unbounded uniform domains with bounded boundaries.
	\\
	
	\noindent {\bf Basic Lemmas:} We now recall some preliminary lemmas from \cite{GS} as well as some simple consequences that will be useful throughout the paper. In what follows, 
	we very often break up $\Omega$ into bands in the following way. 
	
	\begin{definition}
		We set 
		\[
		\Om_0:=\{x\in\Om\, :\, d_\Om(x)\le 1\},
		\]
		and for positive integers $n$ we set
		\[
		\Om_n:=\{x\in\Om\, :\, 2^{n-1}<d_\Om(x)\le2^n\}.
		\]
		Note that $\Om=\bigcup_{n\ge0}\Om_n$.
	\end{definition}
	
	Since $\overline{\Om_\pip}^\pip$ is compact, see Proposition~\ref{prop:compactness},  we have that 
	$\Om_\pip$ is locally compact, and hence a simple topological argument gives the following lemma.
	
	\begin{lemma}\label{lem:geodesic}
		Let $x\in\Om_m$ for some positive integer $m$. 
		Then there is a geodesic in $\Om_\pip$, with respect to the metric $d_\pip$, connecting $x$ and $\infty$.
	\end{lemma}
	
	As a consequence 
	of the above lemma, we can show that every pair of points in
	$\bigcup_{n\ge 1}\overline{\Om_n}$ can be connected in $\Om_\pip$ by a $d_\pip$-geodesic curve.
	
	We next show that $\partial\Om$ being bounded implies that consecutive bands have comparable measure. This follows 
	from the fact that each band has bounded diameter. Indeed, setting 
	$L=\diam_d(\partial\Om)$, it follows from the triangle inequality that $2^{n-1}\le \diam_d(\Om_n)\le 2^{n+1}+L$ 
	for each non-negative integer $n$, and so we can find
	a constant $C_L>0$ such that 
	\begin{equation}\label{eq:diam-bd}
		C_L^{-1}\,2^n\le \text{diam}_d (\Om_n)\le C_L\, 2^n.
	\end{equation}
	
	\begin{lemma}\label{lem:comparable-layers}
		Let $n$ be a non-negative integer. Then there exists a  constant $C_0>0$ such that
		\begin{equation}\label{eq:comp-layers}
			C_0^{-1}\, \mu(\Om_n)\le  \mu(\Om_{n+1})\le C_0\, \mu(\Om_n),
		\end{equation}
		where $C_0$ depends solely on the doubling constant of $\mu$ and the constant $C_L$. Moreover, there exists 
		a $y_n\in\Om_n$ such that $\mu(\Om_n)\approx \mu(B_d(y_n,2^n))$.
	\end{lemma}
	
	\begin{proof}
		Take $x\in\Om_{n+1}$ for which $d_\Om(x)=\tfrac{3}{2}\, 2^{n}$ and $\gamma$ a 
		uniform curve
		with respect to the metric $d$
		with one end point in $\partial\Om$ and the other at $x$. The existence of such a curve is guaranteed by the uniformity of
		$\Om$ with respect to $d$. We can then find $y_n$ in the trajectory of $\gamma$ so that 
		$y_n\in\Om_n$ with $d_\Om(y_n)=2^n$, and so, by~\eqref{eq:diam-bd} and the doubling property of $\mu$,
		\[
		\mu(\Om_n)\le \mu(B_d(y_n,C_L\, 2^n))\lesssim \mu(B_d(y_n,2^n))\lesssim\mu(B_d(x,2^n/C_L))\lesssim \mu(\Om_{n+1}).
		\]
		A similar argument gives us the opposite direction.
	\end{proof}
	
	\begin{lemma}[Lemma~2.10 of~\cite{GS}]\label{lem:dist-to-infty-2}
		Let $x\in\Om_m$ for some integer $m\geq n_0+2$. Then 
		\[
		\frac{5}{11}\sum_{n=m+1}^\infty 2^n\pip(2^n)\le d_\pip(x,\infty)\le C_UC_\pip\sum_{n=m-n_0}^\infty2^n\pip(2^n).
		\]
	\end{lemma}
	
	Thanks to the above lemma, we know that $(\Om_\pip,d_\pip)$ is bounded. Under the additional conditions 
	imposed on $\pip$ in this paper (in particular, conditions (3) and (5) in Definition~\ref{def:phi} above) as compared to~\cite{GS}, we obtain the following.
	
	\begin{lemma} \label{lem:dist-to-infty}
		There exists a constant $\kappa>1$ such that 
		for all non-negative integers $m$ and $x\in\Om_m$, we have
		\[
		\kappa^{-1}\, 2^m\pip(2^m)\le d_\pip(x,\infty)\le \kappa\, 2^m\pip(2^m).
		\]
		Here $\kappa$ depends only on $C_U, C_\pip$, and the implied constant from~\eqref{additionalpip}.
	\end{lemma}
	
	\begin{proof}
		We begin by assuming that $m\geq n_0+2$ and show that it follows from Lemma~\ref{lem:dist-to-infty-2} that
		\begin{equation}\label{disttoinfty}
			d_\pip(x,\infty)\approx\sum_{n=m}^\infty 2^n\, \pip(2^n),
		\end{equation}
		where the implicit constants depend only on $C_U$, $C_\pip$, and $n_0$. 
		Indeed, since $\pip$ is decreasing and satisfies Condition~(3) of Definition~\ref{def:phi}, it follows that
		\[
		\sum_{n=m-n_0}^{m-1}2^n\pip(2^n)\leq n_02^{m-1}\pip(2^{m-n_0})\leq \frac{n_0}{2}C_\pip^{n_0}2^{m}\pip(2^{m}),
		\]
		and so
		\begin{align*}
			\sum_{n=m-n_0}^\infty2^n\pip(2^n)&\leq \frac{n_0}{2}C_\pip^{n_0}2^{m}\pip(2^{m})+\sum_{n=m}^\infty2^n\pip(2^n)\\
			&\leq\left[\frac{n_0}{2}C_\pip^{n_0}+1\right] \sum_{n=m}^\infty2^n\pip(2^n).
		\end{align*}
		
		Similarly,
		\[
		\sum_{n=m}^\infty2^n\pip(2^n)\leq \frac{C_\pip}{2}\sum_{n=m+1}^{\infty}2^m\pip(2^n),
		\]
		and so \eqref{disttoinfty} follows from Lemma~\ref{lem:dist-to-infty-2}. The desired result then follows from \eqref{additionalpip}.
		
		Finally, if $m\le n_0+1$, then $1\gtrsim d_\pip(x,\infty)\gtrsim \pip(2^{n_0+2})\, 2^{n_0}$, and so again the above
		inequality holds even if $m\le n_0+1$.
	\end{proof}
	
	\begin{lemma}[Lemma~2.8 of~\cite{GS}]\label{lem:nearby-points}
		Let $x\in\Om_m$ for some non-negative integer $m$. There exist constants $C_A\geq 1$ 
		and $0<c<1$, depending solely on $C_\pip$ and $C_U$,
		such that if $y\in\Om$ satisfies 
		$d_\pip(x,y)<c\, \pip(2^m)\, 2^{m}$, then 
		\[
		C_A^{-1}\pip(2^m)\, d(x,y)\le d_\pip(x,y)\le C_A\pip(2^m)\, d(x,y). 
		\]
	\end{lemma}

	\section{Background related to metric measure spaces}\label{sec:background}
	
	In this section, we give the definitions of the notions associated with measures and first order calculus
	in metric measure spaces. Namely, we give the definition of doubling measures, first-order calculus on metric measure
	spaces using the approach of upper gradients, and then discuss associated Poincar\'e inequalities. We also
	discuss moduli of families of curves, and variational capacities related to the first-order calculus.
	
	In this section we let $U$ be an
	open and connected subset of a complete metric measure space $(Z,d_Z,\mu_Z)$. 
	In the rest of the paper, $U$ will stand at various points 
	for $\Om$, $\overline{\Om}$, $\Om\cup\{\infty\}$, or $\overline{\Om}\cup\{\infty\}$, while $d_Z$ 
	stands in for either the original metric $d$ or the transformed metric $d_\pip$, and $\mu_Z$ stands 
	in for either the original measure $\mu$ or the transformed measure $\mu_\pip$. 
	Observe that as the completion $\overline{\Om}$ of $\Om$ is the metric space in which $\Om$ is a subset, necessarily
	$\overline{\Om}$ is open in the topology of $\overline{\Om}$. Moreover, as $\Om$ is locally compact, it follows
	that $\Om$ is also open in the topology of $\overline{\Om}$.
	
	Recall that $1\leq p <\infty$ is fixed throughout the paper. 
	
	\begin{deff}\label{loc-doubl}
		We say that $\mu_Z$ is a \emph{locally uniformly doubling} measure on $U$ if 
		$\mu_Z$ is a Radon measure and 
		there is a constant $C_d\ge 1$ and for each
		$x\in U$ there exists $r_x>0$ such that whenever $0<r\le r_x$ we have
		\[
		0<\mu_Z(B(x,2r)\cap U)\le C_d\, \mu_Z(B(x,r)\cap U)<\infty.
		\]
		If there is some constant $A>1$ such that we can choose $r_x=\tfrac{1}{A}\, \dist(x,\partial U)$, then we say that
		$\mu_Z$ is \emph{sub-Whitney doubling} on $U$. 
		We say that $\mu_Z$ is \emph{doubling on} $Z$ if $U=Z$ and $r_x=\infty$ for each $x\in U$. Note that if $\mu_Z$ is doubling on 
		$Z$, then whenever $U\subset Z$ is an open set with $\partial U\ne\emptyset$, we must have that
		$\mu_Z$ is sub-Whitney doubling on $U$.
		A ball $B(x,r)$ with $x\in U$ is said to be a \emph{sub-Whitney ball} if $0<r\le \tfrac1A\, \dist(x,\partial U)$.
	\end{deff}
	
	A metric measure space $(Z,d_Z,\mu_Z)$ with a doubling measure is a \emph{doubling metric space}, that is, there exists 
	some positive integer $N$ such that for each $r>0$ and $x_0\in Z$, and for each $A\subset B(x_0,r)$ 
	such that for each $z,w\in A$ with $z\ne w$ we have $d_Z(z,w)\ge r/2$, then $A$ has at most $N$ 
	number of elements. On the other hand, there are doubling metric spaces that do not support a 
	doubling measure. If the doubling metric space is complete, however, then it does support a 
	doubling measure, see for example~\cite{LuSa, VK}. The completion of a doubling metric space 
	is also doubling, and complete doubling metric spaces are proper (that is, closed and bounded 
	subsets are compact).
	
	\begin{deff}\label{def:mod}
		The $p$-modulus of a collection $\Gamma$ of non-constant, compact, and rectifiable curves in $U$ is
		$$
		\Mod_p(\Gamma;U):=\inf_\rho\!\int_{U}\rho^{\,p}\,d\mu_Z,
		$$
		where the infimum is taken over all admissible $\rho$, that is, all non-negative Borel functions $\rho$ 
		such that $\int_\gamma \rho \,ds\geq{1}$ for each $\gamma\in\Gamma$. A useful property is that 
		$\Mod_p(\Gamma;U)=0$ if and only if there is a non-negative Borel function $\rho\in L^p(U)$ such 
		that $\int_\gamma \rho\, ds=\infty$ for every $\gamma\in\Gamma$, see~\cite{HKST,KoMa}. Note that 
		a countable union of zero $p$-modulus collections of curves is also of $p$-modulus zero. When $U=Z$, 
		we simply write $\Mod_p(\Gamma;U)=\Mod_p(\Gamma)$.
	\end{deff}
	
	In subsequent sections of this paper, $\Mod_p$ will denote the $p$-modulus with 
	respect to the metric $d$ and measure $\mu$, while 
	$\Mod_p^\pip$ will denote the $p$-modulus with respect to the metric $d_\pip$ and measure $\mu_\pip$.
	
	\begin{deff}
		Following~\cite{HK, HKST}, we say that a Borel function $g:U\rightarrow[0,\infty]$ is an \emph{upper gradient} of a function
		$u:U\to\R$ if 
		\[
		|u(y)-u(x)|\le \int_\gamma\! g\, ds
		\]
		whenever $\gamma$ is a non-constant compact rectifiable curve in $U$, 
		with $x$ and $y$ denoting the two end points of $\gamma$
		.
		For $1\le p<\infty$, we say that $g$ is a \emph{$p$-weak upper gradient} 
		of $u$ if the collection $\Gamma$ of non-constant
		compact rectifiable curves for which the above inequality fails is of $p$-modulus zero.
	\end{deff}
	
	It is not difficult to see that if $g_1$ and $g_2$ are both $p$-weak upper gradients of $u$, then so is 
	$\lambda g_1+(1-\lambda)g_2$
	whenever $0\le \lambda\le 1$. Let $D_p(u)$ denote the collection 
	of all $p$-weak upper gradients of $u$; then
	$D_p(u)\cap L^p(U)$ is a closed convex subset of $L^p(U)$, and so if $D_p(u)\cap L^p(U)$ is non-empty, then it
	has a unique element $g_u$ of smallest $L^p$-norm; this function $g_u$ is 
	called the \emph{minimal $p$-weak upper gradient}
	of $u$ in $U$. We invite the interested readers to see~\cite{HKST} for 
	more details on $p$-weak upper gradients.
	
	In subsequent sections of the paper, when the minimal $p$-weak upper gradient of $u$ is 
	taken with respect to the metric $d$, it will be denoted by $g_{u,d}$, while $g_{u,\pip}$ will 
	denote the minimal $p$-weak upper gradient when taken with respect to $d_\pip$.
	
	We say that $u$ is in the \emph{Dirichlet-Sobolev} class $D^{1,p}(U)$ if $D_p(u)\cap L^p(U)$ is non-empty.
	We say that $u$ is in the \emph{Newton-Sobolev} class $N^{1,p}(U)$ if $u\in D^{1,p}(U)$ with 
	$\int_U|u|^p\, d\mu_Z$ finite.
	
	\begin{deff}\label{def:poin}
		We say that $U$ supports a \emph{uniformly local $p$-Poincar\'e inequality} if there are constants $C_P>0$, $\lambda\ge 1$,
		and for each $x\in U$ there exists $r_x>0$,
		such that whenever $0<r\le r_x$ and $u\in D^{1,p}(U)$, we have
		\[
		\vint_{B(x,r)\cap U}\!|u-u_{B(x,r)\cap U}|\, d\mu_Z\le C_P\, r\, \left(\vint_{B(x,\lambda r)\cap U}\!g_u^p\, d\mu_Z\right)^{1/p},
		\]
		where 
		\[
		u_{B(x,r)\cap U}:=\vint_{B(x,r)\cap U}\!u\, d\mu_Z:=\frac{1}{\mu_Z(B(x,r)\cap U)}\int_{B(x,r)\cap U}\!u\, d\mu_Z.
		\]
		Moreover, $U$ supports a \emph{sub-Whitney $p$-Poincar\'e inequality} if there is a constant $A\ge 1$ such that
		for each $x\in U$ we can choose $r_x=\tfrac{1}{A}\, \dist(x,\partial U)$. We say that $Z$ supports a \emph{$p$-Poincar\'e 
			inequality} if $U=Z$ and we can choose $r_x=\infty$ for each $x\in U$.
	\end{deff}
	
	\begin{remark}\label{D=P}
		If $U$ is bounded and supports a uniformly local $p$-Poincar\'e inequality, then $D^{1,p}(U)=N^{1,p}(U)$ as vector spaces,
		but their norms are naturally different. The norm on $N^{1,p}(U)$ incorporates the $L^p$-norm of the function in
		addition to the energy seminorm inherited from $D^{1,p}(U)$; for a function $u\in D^{1,p}(U)$, its energy seminorm is
		$\| u\|_{D^{1,p}(U)}:=\inf_g\left(\int_U g^p\, d\mu_Z\right)^{1/p}$, where the infimum is over all upper gradients $g$ of $u$.
		To turn $D^{1,p}(U)$, with this energy seminorm, into a normed space, one would have to form a quotient space where
		to functions $u_1,u_2\in D^{1,p}(U)$ are said to be equivalent if $\|u_1-u_2\|_{D^{1,p}(U)}=0$; in particular, two functions
		that differ by a constant would have to be considered to be equivalent. We do not wish to do so, and hence 
		$D^{1,p}(U)$ is only a seminormed space.
	\end{remark}
	
	\begin{deff}\label{def:cap}
		Given two sets $E,F\subset U$, the variational $p$-capacity of the condenser $(E, F; U)$ is the number
		\[
		\rcapa_p(E,F;U):=\inf_u \int_Ug_u^p\, d\mu_Z,
		\]
		where the infimum is over all functions $u\in N^{1,p}(U)$ with $u\ge 1$ on $E$ and $u\le 0$ on $F$. When $U=Z$, we 
		simply write $\rcapa_p(E,F;U)=\rcapa_p(E,F)$. 
		
		For a set $A\subset U$, by $\text{Cap}_p(A)$ we mean the Sobolev $p$-capacity 
		\[
		\text{Cap}_p(A):=\inf_u\int_U[|u|^p+g_u^p]\, d\mu_Z,
		\]
		where the infimum is over all functions $u\in N^{1,p}(U)$ with $u\ge 1$ on $A$.
	\end{deff}
	
	Should $U$ support a local $p$-Poincar\'e inequality, then functions in $D^{1,p}(U)$ are necessarily $p$-quasicontinuous.
	Hence, in the above definitions of capacities, we can also insist on the admissible functions $u$ satisfying $u \ge 1$ in a neighborhood of the sets $E$, $A$, respectively, and $u\le 0$ in a neighborhood of the set $F$;
	see for example~\cite{KaS} or~\cite[Theorem~6.11]{BB}. Moreover, by~\cite[Proposition~1.48]{BB}, we have that
	$\text{Cap}_p(A)=0$ if and only if $\mu_Z(A)=0$ and $\Mod_p(\Gamma_A)=0$, where $\Gamma_A$ consists of all
	non-constant compact rectifiable curves in $U$ that intersect $A$.
	
	\begin{remark}\label{rem:capmod}
		By~\cite[Corollary~9.3.2]{HKST}, we know that 
		\[
		\rcapa_p(E,F;U)=\Mod_p(\Gamma(E,F;U);U)=:\Mod_p(E,F;U),
		\] 
		where
		$\Gamma(E,F;U)$ is the collection of all rectifiable curves in $U$ that intersect both $E$ and $F$.
	\end{remark}
	
	As with the modulus, the notation $\rcapa_p, \capa_p$ will be used throughout the paper when 
	taken with respect to  the metric $d$ and measure $\mu$, and $\rcapap,\capa_p^\pip$ when 
	taken with respect to  the metric $d_\pip$ and measure $\mu_\pip$.
	
	Next we recall the definition of $p$-harmonic functions on a metric measure space.
	
	\begin{deff}\label{def:harm}
		A function $u$ on $U$ is said to be a \emph{$p$-minimizer} if $u\in D^{1,p}(U)$ and whenever $v\in D^{1,p}(U)$ has compact
		support $V$ contained in $U$, then
		\[
		\int_V\! g_u^p\, d\mu_Z\le \int_V\! g_v^p\, d\mu_Z.
		\]
		If $(U,d_Z,\mu_Z)$ is locally doubling and locally supports a $p$-Poincar\'e inequality, then there is a locally H\"older
		continuous representative of $u$, see for example~\cite{KS}.
		Continuous $p$-minimizers are called \emph{$p$-harmonic functions}.
	\end{deff}
	
	\begin{deff} \label{def:codim-Haus}
		For $t\geq 0$, the $t$--codimensional Hausdorff measure of a set $A\subset U$ is defined as 
		\[
		\mathcal{H}^{-t}(A;U)=\lim_{\eps\rightarrow{0^+}}\mathcal{H}^{-t}_\eps(A;U),
		\]
		where for each $\eps>0$,
		\[
		\mathcal{H}^{-t}_\eps(A;U)
		=\inf\left\{\sum_{i=1}^{\infty}\frac{\mu_Z(B(x_i,r_i)\cap U)}{r_i^t}\,:\,A\subset\bigcup_{i=1}^{\infty}B(x_i,r_i),\,r_i<\eps \right\}.
		\]
	\end{deff}
	
	\begin{lemma}\label{lem:maximal}
		Let $f\in L^p(U)$, $0< t<p$, and $M>0$. If $\mu_Z$ is locally uniformly doubling on $U$, then $\mathcal{H}^{-t}(E_M)=0$, where
		\[
		E_M= \left\{x\in U: \limsup_{r\rightarrow{0^+}}r^t\vint_{B(x,r)\cap U}\!|f|^p\,d\mu_Z
		>M^p \right\}.
		\]
	\end{lemma}
	
	In the above lemma, the conclusion is valid even if $t\ge p$, but in our use of this lemma in
	the proof of Proposition~\ref{prop:cap-haus} below we require that $t<p$. However, when $t$ is larger than the
	lower mass bound exponent of $\mu_Z$ as discussed in Section~\ref{Sect:Dirichlet}, then 
	it can be shown from the fact that $f\in L^p(U)$ that $E_M$ is empty.
	
	\begin{proof}
		Fix $\eps>0$. From continuity of the integral there exists $\delta>0$ such that for all measurable $V\subset U$,
		$$
		\mu_Z(V)<\delta \implies \int_{V}\!|f|^p\,d\mu_Z<\eps.
		$$
		By the Lebesgue differentiation theorem and the fact that $t>0$, we see
		that $\mu_Z(E_M)=0$ and so there 
		exists an open set $W\subset U$ with $E_M\subset W$ for which $\mu_Z(W)<\delta$. It follows that 
		$\int_{W\cap U}\!|f|^p\,d\mu_Z<\eps$.
		
		We construct a cover of $E_M$ by balls in the following way. For each $x\in E_M$, select $r_x>0$ such that
		\begin{enumerate}
			\item $0<r_x<\eps/5$,
			\item $B(x,5r_x)\subset W$,
			\item $\displaystyle r_x^t\vint_{B(x,r_x)\cap U}\!|f|^p\,d\mu_Z>M^p.$
		\end{enumerate}
		This follows from the definition of $E_M$ and the fact that $W$ is open. An application of the basic $5r$-covering lemma
		(see for example~\cite[Theorem~1.2]{Hei})  
		yields a countable pairwise disjoint subcollection $\{B(x_i,r_i)\}$ such that $E_M\subset \bigcup B(x_i,5r_i)$. Hence,
		\begin{align*}
			\mathcal{H}^{-t}_{\eps}(E_M)
			&\leq\sum_i \frac{ \mu_Z(B(x_i,5r_i)) }{(5r_i)^t}\leq \frac{C_d^{3}}{5^t}\sum_i \frac{ \mu_Z(B(x_i,r_i)) }{r_i^t}
			\\&< \frac{C_d^{3}}{5^tM^p}\int_{W\cap U}\!|f|^p\,d\mu_Z<\frac{C_d^{3}}{5^tM^p}\eps.
		\end{align*}
		The result follows by sending $\eps\rightarrow{0^+}$.  
	\end{proof}
	
	The following proposition 
	relates the $p$-capacity of a set to its codimensional Hausdorff measure.  In the Euclidean setting,
	the following proposition can be found in~\cite[Section~4.7.2, Theorem~4]{EG}.
	
	\begin{proposition}\label{prop:cap-haus}
		Let $U$ support a uniformly local $p$-Poincar\'{e} inequality, and let $\mu_Z$ be a 
		locally uniformly doubling measure on $U$. If $\capa_p(A)=0$ for $A\subset U$, then 
		$\mathcal{H}^{-t}(A)=0$ for all $0< t<p$. 
	\end{proposition}
	
	\begin{proof}
		If $\capa_p(A)=0$, then for each $k\in\N$ there exists a function $u_k\in N^{1,p}(U)$ such that $u_k\geq 1$ 
		on a neighborhood of $A$, $0\le u_k\le 1$ on $U$, and 
		\[
		\int_{U}\![u_k^p+g_k^p]\,d\mu_Z<\frac{1}{2^{kp}}, 
		\]
		where $g_k:=g_{u_k}$. Define $u=\sum_k u_k$. Then $g=\sum_k g_k$ is a $p$-weak upper gradient of $u$ 
		with
		\[
		\left( \int_U\![u^p+g^p]\,d\mu_Z \right)^{1/p}\leq\sum_k\left( \int_U\![u_k^p+g_k^p]\,d\mu_Z \right)^{1/p} <\sum_k\frac{1}{2^k}<\infty.
		\]
		It follows that $u\in N^{1,p}(U)$. 
		
		Since each $u_k$ is at least 1 on a neighborhood of $A$, we have that for $M\geq 1$, $A\subset\{u\geq M\}^\circ$. 
		Hence, for $x\in{A}$ there exists $r_x>0$ such that $B(x,r_x)\subset\{u\geq M\}$, and so 
		$u_{B(x,r)\cap U}\geq{M}$ for $0<r\leq r_x$. Since this is true for all $M\geq{1}$, it follows that 
		$u_{B(x,r)\cap U}\rightarrow\infty$ as $r\rightarrow{0^+}$ for each $x\in A$.
		
		Fix $x\in A$ and assume that
		\[
		\limsup_{r\rightarrow{0^+}}r^t\vint_{B(x,\lambda r)\cap U}\!g^p\,d\mu_Z<\infty.
		\]
		Then for all $0<r\leq 1$ there is some $M\geq 1$ for which
		\[
		r^t\vint_{B(x,\lambda r)\cap U}\!g^p\,d\mu_Z\leq M^{p}.
		\]
		By the Poincar\'{e} inequality, 
		\[
		\vint_{B(x,r)\cap U}\!|u-u_{B(x,r)\cap U}|\, d\mu_Z
		\leq C_P\, r\, \left(\vint_{B(x,\lambda r)\cap U}\!g^p\, d\mu_Z\right)^{1/p}\!\leq C_PMr^{1-t/p}.
		\]
		It follows then that 
		\[
		| u_{B(x,r/2)\cap U} - u_{B(x,r)\cap U} |\leq C_d\,C_P\,Mr^{1-t/p}
		\]
		and so, for $k>j$, 
		\[
		| u_{B(x,r/2^k)\cap U} - u_{B(x,r/2^j)\cap U} |\leq C_d\,C_P\,Mr^{1-t/p} \sum_{i=j+1}^{k}2^{(i-1)(1-t/p)}.
		\]
		As $0\leq t<p$, $1-\frac{t}{p}>0$ and so this is the tail of a convergent geometric series. 
		From this we have that $\{u_{B(x,r/2^k)\cap U}\}$ is a Cauchy sequence in $\R$, contradicting the fact that 
		$u_{B(x,r)\cap U}\rightarrow\infty$ as $r\rightarrow{0^+}$. Therefore,
		\[
		\limsup_{r\rightarrow{0^+}}r^t\vint_{B(x,r)\cap U}\!g^p\,d\mu_Z=\infty.
		\]
		Since this is true for each $x\in A$, we have that $A\subset E_M$ for any $M\geq{1}$, 
		each of which has $t$-codimensional Hausdorff measure zero by Lemma~\ref{lem:maximal}. The result follows. 
	\end{proof}
	
	We end this section by defining two notions that are tools in the study of potential theory.
	
	\begin{deff}\label{Besov}
		For $\alpha>0$ and $1\le p<\infty$, we set the Besov space $B^\alpha_{p,p}(Z)$ to be the class of all functions $f\in L^1_{loc}(Z)$
		such that 
		\[
		\Vert f\Vert_{B^\alpha_{p,p}(Z)}^p
		:=\int_Z\int_Z \frac{|f(y)-f(x)|^p}{d_Z(x,y)^{\alpha p}\, \mu_Z(B(x,d_Z(x,y)))}\, d\mu_Z(y)\, d\mu_Z(x)<\infty.
		\]
	\end{deff}
	
	It was shown in~\cite{CKKSS} that functions in the Besov class are in $L^p(Z,\mu_Z)$ if $Z$ is bounded. Besov spaces
	arise naturally as the trace class of Sobolev spaces. While this is well-known in the setting of Euclidean spaces,
	see for example~\cite{JW} or \cite[Chapter~10]{Maz}, the following extension to the setting of metric spaces is found 
	in~\cite[Theorem~1.1]{Maly}.
	
	\begin{proposition}\label{prop:Besov-trace}
		Suppose that $(Z,d_Z,\mu_Z)$ is doubling and supports a $p$-Poincar\'e inequality for some $1\le p<\infty$, and let
		$U\subset Z$ be a bounded uniform domain such that $\mu_Z\vert_U$ is also doubling. 
		Suppose that there is a measure $\nu$ on $\partial U$ and some positive 
		$\theta<p$ such that for each $w\in\partial U$ and $0<r\le \diam\partial U$, we have
		\[
		\nu(B(w,r)\cap\partial U)\approx \frac{\mu_Z(B(w,r)\cap U)}{r^\theta}.
		\]
		Then there is a bounded linear surjective operator $T:N^{1,p}(U)\to B^{1-\theta/p}_{p,p}(\partial U)$ such that
		for $\nu$-a.e.~$x\in\partial U$,
		\[
		\lim_{r\to 0^+} \vint_{B(x,r)\cap U}|u-Tu(x)|^p\, d\mu_Z=0.
		\]
		Moreover, there is a bounded linear operator $E:B^{1-\theta/p}_{p,p}(\partial U)\to N^{1,p}(U)$ such that
		$T\circ E$ is the identity operator on $B^{1-\theta/p}_{p,p}(\partial U)$.
	\end{proposition}
	
	\begin{deff}\label{Riesz}
		For a non-negative function $u$, we define the \emph{Riesz potential} of $u$ relative to $U$ as 
		$$
		I_{1,U}u(x)=\int_{U}\!\frac{u(y)d_Z(x,y)}{ \mu_Z( B(x,d_Z(x,y)) ) }\,d\mu_Z(y). 
		$$
	\end{deff}
	
	The following proposition is an application of~\cite[Corollary~4.2]{Mak} to the setting where $X=\overline{U}$ and
	$\nu$ is the measure, as in Proposition~\ref{prop:Besov-trace} above, supported on $\partial U$. In this proposition, 
	$Q_Z^-$ plays the role of the lower mass bound exponent for the measure $\mu_Z\vert_U$:
	\[
	\left(\frac{r}{R}\right)^{Q_Z^-}\lesssim \frac{\mu_Z(B(x,r)\cap U)}{\mu_Z(B(x,R)\cap U)}
	\]
	for all $x\in U$ and  $0<r<R<\infty$. 
	
	\begin{proposition}\label{prop:Riesz-neu}
		With $U$ and $\nu$ as in Proposition~\ref{prop:Besov-trace}, and let $1<\bar{p}<\bar{q}$ such that
		\[
		Q^-_Z+\bar{q}-\frac{Q^-_Z\ \bar{q}}{\bar{p}}=\theta.
		\]
		Then there is a constant $C\ge 1$ such that for all balls $B\subset Z$ centered at points in $\overline{U}$,
		setting $B_0:=B\cap\overline{U}$ and all $f\in L^{\bar{p}}(B_0,\mu_Z)$, 
		we have
		\[
		\left(\int_{B_0}(I_{1,B_0}|f|)^{\bar{q}}\, d\nu\right)^{1/\bar{q}}
		\le C\, \mu_Z(B_0)^{1/\bar{q}-1/\bar{p}}\rad(B)^{1-\theta/\bar{q}}\left(\int_{B_0}|f|^{\bar{p}}\, d\mu_Z\right)^{1/\bar{p}}.
		\]
	\end{proposition}

	\section{Doubling property of $\mu_\pip$}\label{sec:doubling}

	In this section we establish the doubling property of the measure $\mu_\pip$.
	Recall the constant $\kappa>1$ established in Lemma~\ref{lem:dist-to-infty}.
	Since it follows from~\eqref{finite} that $\sum_{n=1}^\infty 2^n\pip(2^n)$ is finite, we can find $r_0>0$ 
	such that whenever $m$ is a positive integer with
	$2^m\pip(2^m)\le\kappa\, r_0$ we have that $m>n_0+2$. Here, as in Section~2, $n_0$ is the positive integer 
	satisfying $2^{n_0-1}\le C_U<2^{n_0}$.
	
	From Lemma~\ref{lem:dist-to-infty} we know that if $x\in B_{d_\pip}(\infty, r)\setminus\{\infty\}$ 
	with $r\le r_0$, then necessarily $x\in\Om_m$
	for some $m>n_0+2$.

	\begin{lemma}\label{lem:balls-at-infty}
		For $0<r<r_0$,
		\[
		\mu_\pip(B_{d_\pip}(\infty,r))\le\left[MC_\pip^{Mp}\,C_0^{M+1}+1\right]\sum_{n=m}^{\infty}\pip(2^n)^p\,\mu(\Om_n),
		\]
		and 
		\[
		\mu_\pip(B_{d_\pip}(\infty,r))\ge\left[MC_\pip^{Mp}\,C_0^{M+1}+1\right]^{-1}\sum_{n=m}^{\infty}\pip(2^n)^p\,\mu(\Om_n),
		\]
		where $M= \frac{\log(\kappa^2 )}{\log(\tau/2)}$ and $m$ is any non-negative integer such that 
		for some $x\in\Om$ with $d_\pip(x,\infty)=r$ we have $x\in\Om_m$. Moreover, if $k$ is also a
		non-negative integer such that $\Om_k$ contains a point $z$ with $d_\pip(z,\infty)=r$, then $|k-m|\le M$. 
	\end{lemma}
	
	\begin{proof}
		Let $m_1$ be the smallest non-negative integer such that $2^{m_1}\pip(2^{m_1})\le \kappa r$ and $m_2$ 
		be the largest non-negative integer such that $\kappa 2^{m_2}\pip(2^{m_2})\ge r$. As $r\le r_0$, we have that
		$m_1> n_0+2$ and $m_2\ge n_0+2$.
		Since for each $m\ge 1$,
		every point in $\Om_m$ can be connected to $\infty$ by a $d_\pip$-geodesic by Lemma~\ref{lem:geodesic},
		there is some 
		$x\in\Om$ such that $d_\pip(x,\infty)=r$. With $m$ a positive integer such that $x\in\Om_m$,
		by Lemma~\ref{lem:dist-to-infty} 
		we know that $m_1\le m\le m_2$. 
		From Lemma~\ref{lem:dist-to-infty} again, we have that
		\[
		B_{d_\pip}(\infty, r)\subset \bigcup_{n=m_1}^\infty\Om_n \ \text{ and }
		\bigcup_{n=m_2}^\infty\Om_n\subset B_{d_\pip}(\infty,r),
		\]
		from where it follows from the construction of $\mu_\pip$ and Condition (3) that
		\begin{equation}\label{eq:B1}
			\mu_\pip(B_{d_\pip}(\infty,r))\le \sum_{n=m_1}^\infty \mu_\pip(\Om_n)\le C_\pip^p\sum_{n=m_1}^\infty \pip(2^n)^p\,\mu(\Om_n)
		\end{equation}
		and, this time from the fact that $\pip$ is decreasing, that
		\begin{equation}\label{eq:B2}
			\mu_\pip(B_{d_\pip}(\infty,r))\ge \sum_{n=m_2}^\infty \mu_\pip(\Om_n)\ge \sum_{n=m_2}^\infty \pip(2^n)^p\,\mu(\Om_n).
		\end{equation}
		We now estimate $m-m_1$ and $m_2-m$. Invoking Lemma~\ref{lem:dist-to-infty}, we have that 
		\[
		\kappa^2\, 2^m\pip(2^m)\geq 2^{m_1}\pip(2^{m_1}),
		\]
		and so 
		\[
		\kappa^{2}2^{m-m_1}\pip(2^m)\geq\pip(2^{m_1})\geq \tau^{m-m_1}\pip(2^m),
		\]
		from where it follows that (recall that $\tau>2$)
		\[
		0\le m-m_1\leq\frac{\log(\kappa^2)}{\log(\tau/2)}=M.
		\]
		Similarly, we have $0\le m_2-m\leq M$. Now by combining~\eqref{eq:comp-layers} with~\eqref{eq:B1} we obtain
		\begin{align*}
			\mu_\pip(B_{d_\pip}(\infty,r))&\le \sum_{n=m}^\infty\pip(2^n)^p\mu(\Om_n)+\sum_{n=m_1}^{m-1}\pip(2^n)^p\mu(\Om_n)\\
			&\le \sum_{n=m}^\infty\pip(2^n)^p\mu(\Om_n)+M\, C_\pip^{Mp}\pip(2^m)^p\, C_0^{M+1}\mu(\Om_m)\\
			&\le \left[1+M\, C_\pip^{Mp}\, C_0^{M+1}\right]\sum_{n=m}^\infty\pip(2^n)^p\mu(\Om_n).
		\end{align*}
		By combining~\eqref{eq:comp-layers} with~\eqref{eq:B2} instead, we obtain
		\begin{align*}
			\sum_{n=m}^\infty\pip(2^n)^p\mu(\Om_n)&\le \sum_{n=m_2}^\infty\pip(2^n)^p\mu(\Om_n)+\sum_{n=m}^{m_2-1}\pip(2^n)^p\mu(\Om_n)\\
			&\le  \sum_{n=m_2}^\infty\pip(2^n)^p\mu(\Om_n)+M\,C_\pip^{Mp}\pip(2^{m_2})^p\, C_0^{M+1}\mu(\Om_{m_2})\\
			&\le \left[1+M\, C_\pip^{Mp}\, C_0^{M+1}\right]\sum_{n=m_2}^\infty\pip(2^n)^p\mu(\Om_n)\\
			&\le \left[1+M\, C_\pip^{Mp}\, C_0^{M+1}\right]\, \mu_\pip(B_{d_\pip}(\infty,r)),
		\end{align*} 
		completing the proof.
	\end{proof}
	
	We are now ready to prove the doubling property of $\mu_\pip$, and we do so via the following series of lemmata. The first
	lemma deals with balls centered at $\infty$, the next two lemmata deal with balls that are far away from $\infty$, and
	the final lemma deals with intermediate balls. 
	
	\begin{lemma}\label{lem:balls-at-infinity}
		For $0<r\le r_0/2$, we have
		\[
		\mu_\pip(B_{d_\pip}(\infty,2r))\leq C_1\,\mu_\pip(B_{d_\pip}(\infty,r)),
		\]
		where 
		\[
		C_1=\left[MC_\pip^{Mp}\,C_0^{M+1}+1\right]^2
		\left[\widehat{M}C_\pip^{\widehat{M}p}C_0^{\widehat{M}+1}+1\right]
		\]
		with
		\[
		\widehat{M}=\frac{\log(2\kappa^2)}{\log(\tau/2)},
		\]
		$M$ is as in Lemma~\ref{lem:balls-at-infty}, and $\kappa>1$ is as in Lemma~\ref{lem:dist-to-infty}.
	\end{lemma}
	
	\begin{proof}
		Let $m, \widehat{m}$ be the largest positive integers such that there is some $x\in\Om_m$ and $y\in\Om_{\widehat{m}}$
		with $d_\pip(x,\infty)=r$ and $d_\pip(y,\infty)=2r$. Note then that $m\ge \widehat{m}$.
		Moreover, by the choice of $r_0$, we know that $\widehat{m}\ge n_0+2$.
		By Lemma~\ref{lem:balls-at-infty}, we have  
		\[
		\left[MC_\pip^{Mp}\,C_0^{M+1}+1\right]\mu_\pip(B_{d_\pip}(\infty,r))\geq \sum_{n=m}^\infty\pip(2^n)^p\, \mu(\Om_n)
		\]
		and
		\[
		\mu_\pip(B_{d_\pip}(\infty,2r))\leq 
		\left[MC_\pip^{Mp}\,C_0^{M+1}+1\right] 
		\sum_{n=\widehat{m}}^\infty\pip(2^n)^p\, \mu(\Om_n).
		\] 
		We now use the fact that $m\ge n_0+2$ to
		estimate the size of $m-\widehat{m}$ in the same way as we estimated 
		$m-m_1$ and $m_2-m$ in the proof of Lemma~\ref{lem:balls-at-infty}. From Lemma~\ref{lem:dist-to-infty} it follows that
		\[
		\kappa 2^m\, \pip(2^m)\geq r = \frac12(2r)\geq \frac{1}{2\kappa}2^{\widehat{m}}\, \pip(2^{\widehat{m}}),
		\]
		and so 
		\[
		2\kappa^{2}\, 2^{m-\widehat{m}}\pip(2^m)\geq\pip(2^{\widehat{m}})\geq \tau^{m-\widehat{m}}\pip(2^m),
		\]
		from where it follows that (recall that $\tau>2$)
		\[
		m-\widehat{m}\leq\frac{\log(2\kappa^2)}{\log(\tau/2)}=\widehat{M}.
		\]
		Thus,  
		\begin{align*}
			\sum_{n=\widehat{m}}^{m-1} \pip(2^n)^p\,\mu(\Om_n)
			&\leq (m-\widehat{m})\pip(2^{\widehat{m}})^p\, C_0^{m-\widehat{m}+1}\mu(\Om_{{m}})\\
			&\leq \widehat{M}\, C_\pip^{\widehat{M}p} \,\pip(2^{m})^p\,C_0^{\widehat{M}+1}\mu(\Om_{m}),
		\end{align*}
		and so 
		\begin{align*}
			\sum_{n=\widehat{m}}^\infty \pip(2^n)^p\,\mu(\Om_n)
			&= \sum_{n=\widehat{m}}^{m-1} \pip(2^n)^p\,\mu(\Om_n) + \sum_{n=m}^{\infty} \pip(2^n)^p\,\mu(\Om_n)\\
			&\leq \left[\widehat{M}C_\pip^{\widehat{M}p}C_0^{\widehat{M}+1}+1\right]\sum_{n=m}^{\infty}\pip(2^n)^p\,\mu(\Om_n).
		\end{align*}
		Combining this with Lemma~\ref{lem:balls-at-infty}, we obtain the desired inequality.
	\end{proof}

	\begin{lemma}\label{lem:large}
		Let $x\in\Omega_m$ for some positive integer $m$. For $0<r<r_0/2$, if $\infty\in B_{d_\pip}(x,r/2)$, then
		\[
		\mu_\pip(B_{d_\pip}(x,2r))\leq C_2\,\mu_\pip(B_{d_\pip}(x,r)),
		\]
		where $C_2$ depends on the constant $C_1$ from Lemma~\ref{lem:balls-at-infinity} above. 
	\end{lemma}
	
	\begin{proof}
		It follows from $\infty\in B_{d_\pip}(x,r/2)$ that $B_{d_\pip}(\infty,r/2)\subset B_{d_\pip}(x,r)$ and 
		$B_{d_\pip}(x,2r)\subset B_{d_\pip}(\infty,4r)$. Combining this with Lemma~\ref{lem:balls-at-infinity} yields 
		\[
		\mu_\pip(B_{d_\pip}(x,2r))\leq \mu_\pip(B_{d_\pip}(\infty, 4r))\lesssim \mu_\pip(B_{d_\pip}(\infty,r/2))\leq \mu_\pip(B_{d_\pip}(x,r)).
		\]
	\end{proof}
	
	Now we consider balls that are far away from $\infty$.
	
	\begin{lemma}\label{lem:small}
		Let $x\in\Om$ and $0<r\le r_0/2$ such that $\infty\not\in B_{d_\pip}(x, C_*r)$, where $C_*=4\kappa/c$ with $\kappa$ from Lemma~\ref{lem:dist-to-infty} and $c$ from Lemma~\ref{lem:nearby-points}.
		Then 
		\[
		\mu_\pip(B_{d_\pip}(x,2r))\leq C_3\,\mu_\pip(B_{d_\pip}(x,r)),
		\]
		where $C_3$ depends only on the structural constants
		$C_\mu$, $\widehat{M}$ from Lemma~\ref{lem:balls-at-infinity}, and $C_A$
		from Lemma~\ref{lem:nearby-points}.
		
		Moreover, with $m$ a non-negative integer such that $x\in\Om_m$, we have that 
		\[
		B_{d_\pip}(x,2r)\subset B_d(x, 2C_A\pip(2^m)^{-1}r) \ \text{ and }\ 
		B_d(x, C_A^{-1}\pip(2^m)^{-1}r)\subset B_{d_\pip}(x,r).
		\]
		Furthermore, for all $y\in B_{d_\pip}(x,2r)$ we have that $\pip(2^m)\approx\pip(d_\Om(y))$, 
		with comparison constant independent of  $x,y,r,m$.
	\end{lemma}
	
	It follows from the above lemma that $B_{d_\pip}(x,r)$ is a quasiball with respect to the metric $d$, which is to say that there exists a constant $C>0$, independent of $x$ and $r$, such that 
	$ B_{d}(x,C^{-1}r)\subset B_{d_\pip}(x,r)\subset B_{d}(x,Cr)$.
	
	\begin{proof}
		Since $\infty\not\in B_{d_\pip}(x,C_*r)$, from Lemma~\ref{lem:dist-to-infty} we have that $x\in\Om_m$ with 
		\[
		r\le \frac{\kappa}{C_*}\, 2^m\pip(2^m).
		\]
		Note that as $r\le r_0$, by the choice of $r_0$,
		the above inequality holds if $m\le n_0+1$. So we needed Lemma~\ref{lem:dist-to-infty} only for the case that $m\ge n_0+2$.
		By our choice of $C_*$, which is greater than $2\kappa/c$, 
		for all $y\in B_{d_\pip}(x,2r)$,
		\[
		d_\pip(x,y)<2r\le \frac{2\kappa}{C_*}\, 2^m\pip(2^m)<c\, 2^m\pip(2^m),
		\]
		and so it follows from Lemma~\ref{lem:nearby-points} that 
		\begin{equation}\label{eq:temporary}
			\frac{1}{C_A}\, \pip(2^m)\, d(x,y)\le d_\pip(x,y)\le C_A\, \pip(2^m)\, d(x,y).
		\end{equation}
		Therefore,
		\[
		B_{d_\pip}(x,2r)\subset B_d(x, 2C_A\pip(2^m)^{-1}r) \ \text{ and }\ 
		B_d(x, C_A^{-1}\pip(2^m)^{-1}r)\subset B_{d_\pip}(x,r).
		\]
		The latter inclusion is seen by noting that, thanks to Lemma~\ref{lem:nearby-points} and~\eqref{eq:temporary},
		we must have $B_d(x, C_A^{-1}\pip(2^m)^{-1}r)\cap B_{d_\pip}(x,2r)\subset B_{d_\pip}(x,r)$, and as balls 
		with respect to the metric $d$ are
		connected with respect to both metrics $d$ and $d_\pip$, we must have that 
		$B_d(x, C_A^{-1}\pip(2^m)^{-1}r)\subset B_{d_\pip}(x,2r)$ for otherwise, 
		$B_d(x, C_A^{-1}\pip(2^m)^{-1}r)\cap B_{d_\pip}(x,2r)\setminus B_{d_\pip}(x,r)$ would be non-empty.
		
		Suppose first that $m\ge n_0+2$.
		As in the proof of the previous two lemmata, we have that if $y\in B_{d_\pip}(x,r)$ and $\widetilde{m}\in\N$ such that 
		$y\in\Om_{\widetilde{m}}$, then $|m-\widetilde{m}|\le \widetilde{M}$. 
		Indeed, in Lemma~\ref{lem:nearby-points} we can fix the choice of $C_A$ and
		then always make $c$ smaller and thus $C_*$ larger, 
		and so without loss of generality we have that $C_A\kappa/C_*<1/2$.
		We also have from the limitation on $r$ and from~\eqref{eq:temporary} that $d(x,y)<C_A\tfrac{\kappa}{C_*}\, 2^m$.
		It follows from the fact that $x\in\Om_m$ and $y\in\Om_{\widetilde{m}}$ that $d_\Om(x)\ge 2^m$ and $d_\Om(y)\le 2^{\widetilde{m}+1}$;
		so by the triangle inequality, 
		\[
		2^m-2^{\widetilde{m}+1}\le C_A\frac{\kappa}{C_*}\, 2^m<2^{m-1},
		\]
		whence we obtain $2^{m-1}\le 2^{\widetilde{m}+1}$, that is, $m-\widetilde{m}\le 2$. Similarly we have that
		$2^{\widetilde{m}}-2^{m+1}<2^{m-1}$, and so $2^{\widetilde{m}}<2^{m+2}$, that is, $\widetilde{m}-m<2$. So the choice of
		$\widetilde{M}=2$ satisfies the above statement about $|m-\widetilde{m}|$.
		It follows now from the assumptions on $\pip$ that
		$\pip(2^m)\approx\pip(2^{\widehat{m}})$ for all $y\in B_{d_\pip}(x,2r)$, and so by the doubling property of $\mu$,
		\begin{align*}
			\mu_\pip(B_{d_\pip}(x,2r))&\lesssim \pip(2^m)^p\, \mu(B_{d_\pip}(x,2r))\\
			&\lesssim C_\mu^{k_0}\, \pip(2^m)^p\,\mu(B_d(x, C_A^{-1}\pip(2^m)^{-1}r))\\
			& \lesssim C_\mu^{k_0}\, \pip(2^m)^p\,\mu(B_{d_\pip}(x,r))\\
			&\lesssim C_\mu^{k_0}\, \mu_\pip(B_{d_\pip}(x,r)).
		\end{align*}
		Here, $k_0=\log(2C_A^2)$.
		
		Finally, we take care of the case that $m\le n_0+1$. In this case, 
		by Condition~(3) of Definition~\ref{def:phi},
		we have $1\geq \pip(2^m) \geq \pip(2^{n_0+1})\ge C_\pip^{n_0+1}\pip(1)=C_\pip^{n_0+1}$,
		that is, for $m\le n_0+1$ we have that $\pip(2^m)\approx 1$.
		By Lemma~\ref{lem:nearby-points} again, we have that whenever $y\in B_{d_\pip}(x,2r)$, necessarily
		\[
		d(x,y)\le \frac{2C_A}{\pip(2^m)}\, r\le \frac{2C_A}{\pip(2^m)}\, r_0,
		\]
		and so by choosing a positive integer $k_1$ such that $2^{k_1}>\tfrac{2C_A^2}{\pip(2^m)}\, r_0$, we have that
		each $y\in B_{d_\pip}(x,2C_Ar)$ is in some $\Om_n$ with $n<k_1$, and so $\pip(d_\Om(y))\approx 1$ as well. Therefore,
		\begin{align*}
			\mu_\pip(B_{d_\pip}(x,2r))\approx \mu(B_{d_\pip}(x,2r))&\le \mu(B_d(x, 2C_A\pip(2^m)^{-1}r))\\
			&\lesssim \mu(B_d(x, C_A^{-1}\pip(2^m)^{-1}r))\\
			&\lesssim \mu(B_{d_\pip}(x,r))\\
			&\lesssim \mu_\pip(B_{d_\pip}(x,r)),
		\end{align*}
		finishing the proof.
	\end{proof}
	
	Finally, we take care of the intermediate balls.
	In what follows, set
	\begin{equation}\label{eq:Tee}
		T=\frac{2C_A}{c}.
	\end{equation}

	\begin{lemma}\label{lem:intermediate}
		Let $x\in\Om$ and $0<r\le \tfrac{r_0}{8(C_*+1)}$ such that $\infty\in B_{d_\pip}(x,C_*r)\setminus B_{d_\pip}(x,r/2)$, where $C_*$ is as in Lemma~\ref{lem:small}, then 
		\[
		\mu_\pip(B_{d_\pip}(x,2r))\le C_4\, \mu_\pip(B_{d_\pip}(x,r)).
		\]
		Moreover, with $m$ a non-negative integer such that $x\in\Om_m$ and fixing $C_\Lambda\ge 1$,
		independent of $x$, $m$, and $r$, so that
		\[
		\frac{2^m\, \pip(2^m)}{C_\Lambda}\le r\le C_\Lambda\, 2^m\, \pip(2^m),
		\]
		for each $y\in B_{d_\pip}(x,r/(8TC_\Lambda))$ we have $\pip(d_\Om(y))\approx \pip(2^m)$
		and 
		\[
		\mu(\Om_m)\approx \mu(B_d(x,2^m/(8TC_\Lambda^2C_A))),
		\]
		and 
		\[
		B_d(x,2^m/(8TC_\Lambda^2C_A))\subset B_{d_\pip}(x,r/(8TC_\Lambda)).
		\]
	The constant $C_4$ depends only on the structural constants $C_\mu$, $C_\pip$, and the constants $C_0$ from~Lemma \ref{lem:comparable-layers}, $\kappa$ from 
	Lemma~\ref{lem:dist-to-infty}, $C_A$ from Lemma~\ref{lem:nearby-points}, and $M$ from Lemma~\ref{lem:balls-at-infty}.
	\end{lemma}
	
	\begin{proof}
		By our assumptions, $B_{d_\pip}(x,2r)\subset B_{d_\pip}(\infty, (C_*+2)r)$. 
		Therefore
		\[
		\mu_\pip(B_{d_\pip}(x,2r))\le \mu_\pip(B_{d_\pip}(\infty, 2(C_*+1)r)).
		\]
		Since $2(C_*+1)r\le r_0/4$, it follows from Lemma~\ref{lem:balls-at-infty} that with the choice of $m$ so that $x\in\Om_m$,
		\begin{equation}\label{eq:twice-r}
			\mu_\pip(B_{d_\pip}(x,2r))\lesssim \sum_{n=m}^\infty \pip(2^n)^p\mu(\Om_n)\lesssim \pip(2^m)^p\mu(\Om_m),
		\end{equation}
		and as $d_\pip(x,\infty)\ge r/2$, we have that $d_\pip(x,\infty)\approx r$. Now by Lemma~\ref{lem:dist-to-infty} 
		we have that
		\[
		2^m\, \pip(2^m)\approx  d_\pip(x,\infty)\approx r.
		\]
		Thus there is a constant $C_\Lambda>1$, which is independent of $x$, $m$, and $r$, so that
		\[
		\frac{2^m\, \pip(2^m)}{C_\Lambda}\le r\le C_\Lambda\, 2^m\, \pip(2^m).
		\]
		Thus with our choice of $T$ from~\eqref{eq:Tee},
		from Lemma~\ref{lem:nearby-points} we have that if $y\in\Om$ such that $d_\pip(x,y)\le \tfrac{r}{8TC_\Lambda}$, then
		\[
		\frac{d(x,y)}{C_A}\, \pip(2^m)\le d_\pip(x,y)\le C_A\, d(x,y)\, \pip(2^m).
		\]
		It follows that $d(x,y)\le \tfrac{C_A\, 2^m}{8T}$. By our choice of $T$, we know that $T>C_A$, and so we have that
		$d(x,y)<2^{m-3}$. Therefore 
		\[
		d_\Om(y)\ge d_\Om(x)-2^{m-3}\ge 2^{m-2}\ \text{ and }\ d_\Om(y)\le d_\Om(x)+2^{m-3}\le 2^{m+1}.
		\]
		It follows that for each $y\in B_{d_\pip}(x,r/(8TC_\Lambda))$ we have $\pip(d_\Om(y))\approx \pip(2^m)$. Hence
		\[
		\mu_\pip(B_{d_\pip}(x,r/(8TC_\Lambda)))\approx\pip(2^m)^p\, \mu(B_{d_\pip}(x,r/(8TC_\Lambda))).
		\]
		Moreover, if $y\in\Om$ such that $d(x,y)<2^m/(8TC_\Lambda^2C_A)$, then 
		$d_\pip(x,y)<r/(8TC_\Lambda)$, and so $y\in B_{d_\pip}(x,r/(8TC_\Lambda))$. By the doubling property of $\mu$,
		we know that with $\zeta\in\partial\Om$ (and, without loss of generality, assuming that $\diam_d(\partial\Om)\le 1$),
		\[
		\mu(B_d(x,2^m/(8TC_\Lambda^2C_A)))\approx\mu(B_{d}(\zeta,2^{m+1}))
		\approx \mu\left(\bigcup_{n=0}^{m+1}\Om_n\right), 
		\]
		and so we have that
		\[
		\mu(\Om_m)\lesssim\mu(B_d(x,2^m/(8TC_\Lambda^2C_A))).
		\]
		From the above discussion, we now have that
		\begin{align*}
			\pip(2^m)^p\mu(\Om_m)&\lesssim\pip(2^m)^p \mu(B_d(x,2^m/(8TC_\Lambda^2C_A)))\\
			& \lesssim \pip(2^m)^p\, \mu(B_{d_\pip}(x,r/(8TC_\Lambda)))\\
			&\approx \mu_\pip(B_{d_\pip}(x,r/(8TC_\Lambda))).
		\end{align*}
		Combining the above estimate with~\eqref{eq:twice-r} we obtain
		\[
		\mu_\pip(B_{d_\pip}(x,2r)\lesssim \pip(2^m)^p\, \mu(\Om_m)\lesssim \mu_\pip(B_{d_\pip}(x,r/(8TC_\Lambda)))
		\le \mu_\pip(B_{d_\pip}(x,r)),
		\]
		as desired.
	\end{proof}

	The above lemmata together prove that $\mu_\pip$ is a uniformly locally doubling measure 
	on $(\Om_\pip,d_\pip)$. In fact, $\mu_\pip$ is globally doubling as the assumption that $\partial\Om$ is bounded 
	implies the compactness of $\overline{\Om_\pip}^{\pip}=\overline\Om\cup\{\infty\}$ with respect to the metric $d_\pip$. 
	We summarize this in the following theorem.
	
	\begin{theorem}\label{thm:double}
		The metric measure space $(\Om_\pip,d_\pip,\mu_\pip)$ is doubling.
	\end{theorem}
	
	Conversely, $\overline{\Om_\pip}^{\pip}$ cannot be compact with respect to $d_\pip$ (hence $\mu_\pip$ cannot be 
	doubling on $\Om_\pip$) without $\partial\Om$ being bounded, as we will see now.
	
	\begin{proposition}\label{prop:compactness}
		$\overline{\Om_\pip}^{\pip}=\overline\Om\cup\{\infty\}$ is compact with respect to the metric $d_\pip$ if and only if 
		$\partial\Omega$ is bounded with respect to the metric $d$.
	\end{proposition}
	
	\begin{proof}
		Assume that $\partial\Omega$ is bounded with respect to $d$. Consider a sequence 
		$(x_k)\subset \overline{\Om_\pip}^{\pip}$. If infinitely-many of the terms in the sequence equal $\infty$ or if 
		$\liminf d_\pip(x_k,\infty)=0$, then there is a subsequence of $(x_k)$ converging to infinity with respect to $d_\pip$. 
		As such, we assume without loss of generality that $(x_k)\subset\overline\Om$ and that 
		$d_\pip(x_k,\infty)\geq\tfrac{\tau}{2}$ for all $k$, where $\tau:=\liminf d_\pip(x_k,\infty)>0$. It follows from 
		Lemma~\ref{lem:dist-to-infty} that there is some $N_0$ for which
		\[
		x_k\in \overline{\bigcup_{n=0}^{N_0}\Om_n}= \bigcup_{n=0}^{N_0}\overline{\Om_n}
		\]
		for all $k$. Each set $\overline{\Om_n}$ is closed in $\overline\Om$ with respect to $d$ and bounded since 
		$\diam_d(\Om_n)\le 2^{n+2}+\diam_d(\partial\Om)<\infty$. Hence, $\overline{\Om_n}$ is compact with respect 
		to $d$ as $\overline\Om$ is proper. By Lemma~\ref{lem:nearby-points}, $d$ and $d_\pip$ are locally bi-Lipschitz 
		equivalent and so the two metrics are homeomorphic on $\overline{\Om_n}$. Hence, 
		$\bigcup_{n=0}^{N}\overline{\Om_n}$ is compact with respect to $d_\pip$. Therefore, there exists 
		a $d_\pip$-convergent subsequence of $(x_k)$. 
		
		Assume now that $\partial\Omega$ is unbounded with respect to $d$. Then 
		we may find a sequence $(\zeta_k)\subset\partial\Om$ such that 
		$d(\zeta_k,\zeta_{j})\geq \tau\in(0,\tfrac{1}{10}]$ for all $k\neq{j}$. It follows 
		from the local isometry of $d$ and $d_\pip$ on $\partial\Omega$ (see Lemma 2.6 of \cite{GS}) 
		that $d_\pip(\zeta_k,\zeta_j)\geq\tau$ for all $k\neq\ j$, and so $(\zeta_k)$ cannot converge in 
		$\overline\Om$ with respect to $d_\pip$. Moreover, $d_\pip(\zeta_k,\infty)\geq 1$ for all $k$ and 
		so cannot converge to $\infty$ with respect to $d_\pip$.
	\end{proof}

	\section{Uniformity of $\Om_\pip\setminus\{\infty\}$}
	
	In~\cite{GS} it was shown that $\Om_\pip=\Om\cup\{\infty\}$ is a uniform domain with respect to the metric $d_\pip$ and that
	$\partial\Om_\pip=\partial\Om$. It is not always the case that the removal of a point from a uniform domain results in a uniform
	domain, as seen by the example domain $(-1,0]\times[0,1)\cup[0,1)\times(-1,0]\subset X$ where the metric space 
	$X=\R^2\setminus((-1,0)^2\cup(0,1)^2)$. However, in our setting, $\infty\in\Om_\pip$ has a special role, given the fact that
	$\Om$ itself is a uniform domain with respect to the metric $d$. 
	
	In the rest of this section, by increasing the uniformity constant $K$ if needed, we can find uniform curves for which
	every subcurve is also a uniform curve (all with respect to the metric $d_\pip$), see~\cite[Theorem~2.10]{BHK}.
	In the remainder of this section, when
	we choose a $K$-uniform curve with respect to the metric $d_\pip$, we will also implicitly assume that every subcurve is also
	$K$-uniform with respect to $d_\pip$. 
	
	\begin{lemma}\label{lem:ann-qcvx}
		Let $0<r\le r_0/C$ and $x,y\in B_{d_\pip}(\infty,r)\setminus B_{d_\pip}(\infty, r/2)$, where $C=2\kappa^2C_UC_\pip$. Then 
		there is a curve 
		$\beta\subset B_{d_\pip}(\infty, Cr)\setminus B_{d_\pip}(\infty, r/C)$ with end points $x,y$ such that
		$\ell_\pip(\beta)\approx d_\pip(x,y)$. 
	\end{lemma}
	
	\begin{proof}
		Let $m$ be a positive integer such that $x\in\Om_m$. Then by~\eqref{additionalpip} and by Lemma~\ref{lem:dist-to-infty}
		we know that $d_\pip(x,\infty)\approx 2^m\pip(2^m)$. Thus, with $y\in\Om_k$ we also have $2^k\pip(2^k)\approx 2^m\pip(2^m)$.
		It follows that there is some $N_0\in\N$ such that $|k-m|\le N_0$, see the proof of Lemma~\ref{lem:balls-at-infty} above. Moreover,
		by the choice of $r_0$ we also have that $k\ge 2n_0$ and $m\ge 2n_0$.
		
		Let $\beta$ be a $C_U$-uniform curve in $\Om$, with respect to the
		original metric $d$, connecting $x$ to $y$. Then, from~\eqref{eq:diam-bd}  
		we see that $d(x,y)\lesssim 2^m$,
		and so $\ell_d(\beta)\lesssim 2^{m+n_0}$. Hence 
		\[
		\beta\subset\bigcup_{j=m-n_0-N_0}^{m+n_0+N_0}\Om_j,
		\]
		whence we obtain 
		\[
		\ell_\pip(\beta)\approx \pip(2^m)\ell_d(\beta)\approx\pip(2^m)\, d(x,y)\lesssim\pip(2^m)\, 2^m\approx r.
		\] 
		By Lemma~\ref{lem:dist-to-infty} above and our choice of $C$, we have
		\[
		\beta\subset B_{d_\pip}(\infty, Cr)\setminus B_{d_\pip}(\infty, r/C).
		\]
		It also follows that $d_\pip(x,y)\approx \pip(2^m)\, d(x,y)\approx\ell_\pip(\beta)$, thus completing the proof.
	\end{proof}
	
	In what follows, let $K\ge 1$ denote the uniformity constant of $\Om_\pip$. If $\gamma_1$ is a curve with end points $x$ and $y$, and $\gamma_2$ is a curve with end points $y$ and $z$, then we denote by $\gamma_1+\gamma_2$ the concatenation of two curves $\gamma_1$ and $\gamma_2$, having end points $x$ and $z$. 
	
	\begin{theorem}\label{thm:uniform}
		The set $\Om_\pip\setminus\{\infty\}$ is a uniform domain with respect to the metric $d_\pip$.
	\end{theorem}
	
	\begin{proof}
		From~\cite{GS} we know that $(\Om_\pip,d_\pip)$ is a uniform domain.
		Let $x,y\in\Om_\pip\setminus\{\infty\}$ with $x\ne y$, and let $\gamma$ be a $K$-uniform curve in $\Om_\pip$, with the uniformity
		with respect to the metric $d_\pip$, with end points $x,y$.  
		Without loss of generality let $d_\pip(x,\infty)\le d_\pip(y,\infty)$. We now consider two cases.
		\\
		
		\noindent{\bf Case 1:} We suppose first that 
		\[
		\gamma\subset \Omega_\pip\setminus B_{d_\pip}(\infty, d_\pip(x,\infty)/4CK),
		\]
		where $C=2\kappa^2C_UC_\pip$ is from Lemma~\ref{lem:ann-qcvx} above.
		In this case, for $z\in B_{d_\pip}(\infty, 4d_\pip(x,\infty))\cap\gamma$, we have that 
		\[
		\frac{d_\pip(x,\infty)}{4CK} \le d_\pip(z,\infty)< 4d_\pip(x,\infty)
		\]
		and so
		\[
		\ell_\pip(\gamma[x,z])\le K\, d_\pip(x,z)<5K\, d_\pip(x,\infty)\le 20CK^2\, d_\pip(z,\infty).
		\]
		For $z\in \gamma\setminus B_{d_\pip}(\infty, 4d_\pip(x,\infty))$, we have that
		\[
		\ell_\pip(\gamma[x,z])\le K\, d_\pip(x,z)\le K[d_\pip(x,\infty)+d_\pip(z,\infty)]\le 2K\, d_\pip(z,\infty).
		\]
		Combining the above two subcases, we have that for each $z\in\gamma$,
		\[
		d_\pip(z,\infty)\ge \frac{1}{20CK^2}\, \ell_\pip(\gamma[x,z])\ge \frac{1}{20CK^2}\, \min\{\ell_\pip(\gamma[x,z]),\ell_\pip(\gamma[z,y])\}.
		\]
		
		\noindent{\bf Case 2:} If, instead, we have
		\[
		\gamma\cap B_{d_\pip}(\infty, d_\pip(x,\infty)/4CK)\ne \emptyset,
		\]
		then let $w_1, w_2\in \gamma$ such that 
		\[
		\gamma[x,w_1]\cup\gamma[w_2,y]\subset\Om_\pip\setminus B_{d_\pip}(\infty, d_\pip(x,\infty)/2CK)
		\]
		and
		\[
		d_\pip(\infty,w_1)=d_\pip(\infty, w_2)=\frac{d_\pip(x,\infty)}{2CK}.
		\]
		Let $\beta$ be a curve from Lemma~\ref{lem:ann-qcvx} with end points $w_1$, $w_2$. For $z\in\gamma[x,w_1]$,
		a repetition of the argument of Case~1 above yields $d_\pip(z,\infty)\ge \tfrac{1}{20CK^2}\ell_\pip(\gamma[x,z])$.
		For $z\in\beta$, we have that 
		\[
		d_\pip(z,\infty)\ge \frac{d_\pip(x,\infty)}{2C^2K},
		\]
		and
		\begin{align*}
			\ell_\pip(\gamma[x,w_1])+\ell_\pip(\beta[w_1,z])&\le K\, d_\pip(x,w_1)+\ell_\pip(\beta)\\
			&\lesssim d_\pip(x,w_1)+d_\pip(w_1,w_2)\\
			&\lesssim d_\pip(x,\infty)+2d_\pip(w_1,\infty)+d_\pip(w_2,\infty)\\
			&\lesssim d_\pip(x,\infty).
		\end{align*}
		It follows that $d_\pip(z,\infty)\gtrsim \ell_\pip(\gamma[x,w_1]+\beta[w_1,z])$.
		For $z\in\gamma[w_2,y]\cap B_{d_\pip}(\infty, 4d_\pip(x,\infty))$, we have
		that 
		\[
		d_\pip(z,\infty)\ge \frac{d_\pip(x,\infty)}{2CK},
		\]
		and 
		\[
		\ell_\pip(\gamma[x,w_1]+\beta+\gamma[w_2,z])\lesssim  d_\pip(x,\infty)+4d_\pip(x,\infty)\lesssim d_\pip(x,\infty),
		\]
		and therefore 
		\[
		d_\pip(z,\infty)\gtrsim  \ell_\pip(\gamma[x,w_1]+\beta+\gamma[w_2,z]).
		\]
		For $z\in\gamma[w_2,y]\setminus B_{d_\pip}(\infty, 4d_\pip(x,\infty))$, we have
		\begin{align*}
			\ell_\pip(\gamma[w_2,z])\le K\, d_\pip(w_2,z)
			&\le K[d_\pip(w_2,\infty)+d_\pip(z,\infty)]\\
			&=K[d_\pip(x,\infty)/(2CK)+d_\pip(z,\infty)]\\ 
			& \lesssim d_\pip(z,\infty).
		\end{align*}
		Observe also from the above discussion that 
		\[
		\ell_\pip(\gamma[x,w_1]+\beta)\lesssim d_\pip(x,\infty)\lesssim d_\pip(z,\infty).
		\]
		It then again follows that $d_\pip(z,\infty)\gtrsim  \ell_\pip(\gamma[x,w_1]+\beta+\gamma[w_2,z])$.
		Combining the above four possibilities in this case, we obtain for each 
		$z\in \widehat{\gamma}:=\gamma[x,w_1]+\beta+\gamma[w_2,y]$ that
		\[
		d_\pip(z,\infty)\gtrsim \ell_\pip(\widehat{\gamma}[x,z]).
		\]
		
		From Cases~1 and~2 above we see that there is a curve $\widehat{\gamma}$ with end points $x,y$ such that
		for each $z\in\widehat{\gamma}$ we have that
		\[
		d_\pip(z,\infty)\gtrsim \ell_\pip(\widehat{\gamma}[x,z]),
		\]
		and moreover, $\ell_\pip(\widehat{\gamma})\lesssim d_\pip(x,y)$. Here, in Case~1 above, we merely set 
		$\widehat{\gamma}=\gamma$, the original uniform curve with respect to $d_\pip$, connecting $x$ to $y$.
		So in Case~1 we have from the $K$-uniformity of $\gamma$ with respect to $d_\pip$ that
		$\ell_\pip(\widehat{\gamma})\le K\, d_\pip(x,y)$. In Case~2 we have that
		$\ell_\pip(\widehat{\gamma})\le \ell_\pip(\gamma)+\ell_\pip(\beta)\le K\, d_\pip(x,y)+\ell_\pip(\beta)$.
		Moreover, by the choice of $\beta$, we have that $\ell_\pip(\beta)\lesssim d_\pip(w_1,w_2)\le \ell_\pip(\gamma)$
		because $w_1,w_2$ belong to $\gamma$, and hence again we have that
		$\ell_\pip(\widehat{\gamma})\lesssim d_\pip(x,y)$, thus justifying the inequality given above.
		Thus to show that $\widehat{\gamma}$ is a uniform curve in $\Om_\pip\setminus\{\infty\}$, it now only
		remains to check $\dist_\pip(z,\partial\Om)$ for each $z\in\widehat{\gamma}$. 
		
		In the situation considered in Case~1 above, the curve $\widehat{\gamma}$ is also a $K$-uniform curve in
		$\Om_\pip$ with respect to $d_\pip$, and so we immediately have 
		\[
		\text{dist}_\pip(z,\partial\Om)\ge \tfrac{1}{K}\min\{\ell_\pip(\widehat{\gamma}[x,z]), \ell_\pip(\widehat{\gamma}[z,y])\}
		\]
		as desired. In the situation considered in Case~2, the above inequality holds also when $z\in\gamma[x,w_1]\cup\gamma[w_2,y]$.
		We set $r=d_\pip(x,\infty)/2CK$. Note that then $\ell_\pip(\gamma[x,w_1])\ge (2CK-1)r$, and hence by the 
		uniformity of the curve $\gamma$, we have that
		\[
		\text{dist}_\pip(w_1,\partial\Om)\ge \frac{2CK-1}{K}\, r.
		\]
		Hence, for each $z\in\beta$, we obtain
		\begin{align*}
			\text{dist}_\pip(z,\partial\Om)&\ge \text{dist}_\pip(w_1,\partial\Om)-d_\pip(w_1,\infty)-d_\pip(z,\infty)\\
			&\ge \frac{2CK-1}{K}\, r-r-Cr\\
			&=(C-1-K^{-1})\, r.
		\end{align*}
		Here we used the fact that $\beta\subset B_{d_\pip}(\infty, Cr)\setminus B_{d_\pip}(\infty, r/C)$. 
		Note that  $\ell_\pip(\gamma[x,w_1]+\beta)\lesssim r$. As
		$C>2$, we now have the desired inequality $\dist_\pip(z,\partial\Om)\gtrsim \ell_\pip(\gamma[x,w_1]+\beta)$.
		The claim follows. 
	\end{proof}
	
	\begin{remark}
		The proof of the above theorem can be modified to show that if $\Om$ is a uniform domain and $z_0\in\Om$ is such that
		annular quasiconvexity as in Lemma~\ref{lem:ann-qcvx} holds with $z_0$ playing the role of $\infty$, then
		$\Om\setminus\{z_0\}$ is also a uniform domain. 
	\end{remark}
	
	\begin{remark}
		The utility of uniformity of $\Om_\pip\setminus\{\infty\}$ stems from the fact that when we transform $\Om$ under the metric
		$d_\pip$ and the measure $\mu_\pip$, all the functions $u$ in the Dirichlet-Sobolev class $D^{1,p}(\Om)$ 
		belong to the 
		Dirichlet-Sobolev class of the transformed space $D^{1,p}(\Om_\pip\setminus\{\infty\})$; note that as sets,
		$\Om=\Om_\pip\setminus\{\infty\}$. In order to gain control over the behavior of transformed functions in the 
		uniform domain $\Om_\pip$, we need to know that functions in $D^{1,p}(\Om_\pip\setminus\{\infty\})$ also have an 
		extension to $\infty$ that belongs to $D^{1,p}(\Om_\pip)$; see the discussion in Section~\ref{Sect:Dirichlet} 
		and Proposition~\ref{lem:D-to-N} below.
		Knowledge of uniformity of $\Om_\pip\setminus\{\infty\}$, together with the information that 
		$\Om_\pip\setminus\{\infty\}$ satisfies a $p$-Poincar\'e inequality when $\Om$ itself does (see 
		Section~\ref{Sect:PI} below) aids us in this extension.
	\end{remark}

	\section{Poincar\'e inequalities}\label{Sect:PI}
	
	The goal of this section is to demonstrate that $(\Om_\pip,d_\pip,\mu_\pip)$ supports a 
	$p$-Poincar\'e inequality when $(\Om,d,\mu)$ supports a sub-Whitney $p$-Poincar\'e inequality
	as in Definition~\ref{def:poin}.
	This result can be proved using a variant of the Boman chain condition method that Haj\l{}asz and 
	Koskela~\cite{HaK} used to prove that if all balls satisfy a 
	Poincar\'e inequality, then all sets satisfying a chain condition also satisfy a Poincar\'e inequality. 
	We do not know {\it a priori} whether all balls in $(\Om,d_\pip,\mu_\pip)$
	satisfy a Poincar\'e inequality, but all balls in $\Om_\pip$ can be covered with chains of smaller (i.e., sub-Whitney) balls 
	on which $\pip$ is approximately constant and which therefore inherit a Poincar\'e inequality from 
	$(\Om,d,\mu)$. 
	We will also see that these small balls (with respect to the metric $d$) that make up the chain 
	are quasiballs with respect to the metric $d_\pip$, as for example in the proof of Lemma~\ref{lem:small}.
	For the readers' 
	convenience, we provide a complete proof here. Our proof uses an
	application of~\cite[Theorem~4.4]{BS}. 
	We start with the following chain condition.
	
	\begin{definition} \label{def:chain}
		Let $\mathcal{B}$ be a family of balls in a metric measure space $X$ and $\lambda,M\geq 1$ and $a>1$. 
		We say that $A \subset X$ 
		satisfies the chain condition $C(\mathcal B,\lambda,M,a)$ if there exists a distinguished ball $B_0\subset A$ 
		that belongs to $\mathcal B$ such that for every $x\in A$ there exists an infinite sequence of balls 
		$\{B_i\}_{i=0}^\infty\subset \mathcal B$ (called a ``chain") with the following properties:
		\begin{enumerate}
			\item[(i)] $\lambda B_i\subset A$ for $i=0,1,2,\ldots$ and $B_i$ is centered at $x$ for all sufficiently large $i$;
			\item[(ii)] for $i\geq 0$, the radius $r_i$ of $B_i$ satisfies 
			$M^{-1}(\diam A)\,a^{-i}\leq r_i\leq M (\diam A)\,a^{-i}$; and,
			\item[(iii)] the intersection $B_i\cap B_{i+1}$ contains a ball $B_i'$ such that $B_i\cup B_{i+1}\subset MB_i'$ for all $i\geq 0$.
		\end{enumerate}
	\end{definition}
	
	In this section, we consider sub-Whitney balls corresponding to the constant $2\lambda M$, see
	Definition~\ref{loc-doubl}.

	\begin{remark}\label{rem:BS-chain}
		Given a uniform domain $\Om$, we set the collection $\mathcal{B}$ to consist of balls centered at points in $\Om$ and
		with radii such that the ball is also contained in $\Om$. Then \cite[Lemma~4.3]{BS} tells us that when $\Om$ is a uniform
		domain, there are constants $a$, $\lambda$ and $M$ such that for each 
		$x\in\Om$ and $r>0$, the set $B(x,r)\cap\Om$ satisfies the chain condition $C(\mathcal B, \lambda, M, a)$
		with $x$ the center of the distinguished ball $B_0$.
		The above chain condition is equivalent to the chain condition given in~\cite[Lemma~4.3]{BS}. While the chain of balls in~\cite{BS}
		are not of strictly dyadically decreasing radii, there are at most $L$ balls of the same radius, with $L$ depending
		solely on the uniformity constant of $\Om$ and the choice of $\lambda$; 
		hence, for sufficiently large $M$ in our definition above, the chains constructed in
		\cite[Lemma~4.3]{BS} satisfy the conditions in Definition~\ref{def:chain}. Therefore we can exploit 
		\cite[Theorem~4.4]{BS}.
	\end{remark}

	\begin{theorem}\label{thm:PI-no-infty}
		The uniform domains $\Om_\pip\setminus\{\infty\}$ and $\Om_\pip$, as well as $\overline{\Om_\pip}^\pip$,
		equipped with the metric $d_\pip$ and the measure $\mu_\pip$,
		all satisfy $p$-Poincar\'e inequality if $(\Om, d,\mu)$ satisfies a sub-Whitney $p$-Poincar\'e inequality.
	\end{theorem}
	
	\begin{proof}
		We wish to choose $\sigma>1$ in applying~\cite[Lemma~4.3]{BS} such that when
		$x\in\Om_m$ and $r>0$ such that $\infty\not\in B_{d_\pip}(x,4\sigma r)$, then $\infty\not\in B_{d_\pip}(x,5C_A^2\lambda r)$ and 
		$4C_A^2\lambda r<c\, 2^m\pip(2^m)$, where $c, C_A$ are the constants from Lemma~\ref{lem:nearby-points}. We do this as follows.
		
		Clearly the first condition is satisfied if $B_{d_\pip}(x,5C_A^2\lambda r) \subset B_{d_\pip}(x,4\sigma r)$, that is, 
		whenever $\sigma\geq 5C_A^2\lambda/4$.  To make sure that the second condition is also satisfied, we
		only consider the radii $r$ for which $\infty\not\in B_{d_\pip}(x,4\sigma r)$, that is, 
		$d_\pip(x,\infty)\ge 4\sigma r$. 
		Combining this with Lemma~\ref{lem:dist-to-infty}, we obtain
		\[
		4\sigma r \leq d_\pip(x,\infty)\leq \kappa\, 2^m\pip(2^m).
		\]
		If $\sigma\ge C_A^2\lambda \kappa/c$, the above inequality implies that the second condition 
		$4C_A^2\lambda r<c\, 2^m\pip(2^m)$ is satisfied. 
		Henceforth, we fix a choice of
		\[
		\sigma>\max\bigg\lbrace \frac{5C_A^2\lambda}{4},\, \frac{C_A^2\lambda\kappa}{c}\bigg\rbrace.
		\]
		
		For each ball $B_{d_\pip}$ in $\Om_\pip\setminus\{\infty\}$ (that is, it is the intersection of a ball in
		$\overline{\Om_\pip}$ with $\Om_\pip\setminus\{\infty\}$), we can appeal to ~\cite[Lemma~4.3]{BS} to construct
		chains of balls $B_i=B_{d_\pip}(x_i,r_i)$, $i\in\N$, corresponding to the above choice of $\sigma$. 
		From the above discussion
		and by Lemma~\ref{lem:nearby-points}, we have that 
		\[
		B_i\subset B_d(x_i,\tfrac{C_A}{\pip(2^m)}r_i)\subset C_A^2B_i\subset\Om_\pip\setminus\{\infty\}. 
		\]
		As $C_A^2B_i\subset B_{d_\pip}$ and $\infty\not\in B_{d_\pip}$, the last inclusion above holds.
		Moreover, the weight $\pip(d_{\Omega}(y))^p$ is approximately constant on 
		$C_A^2B_i$, with the comparison constant independent of the ball. Therefore,
		\[
		u_{B_d\big(x,\tfrac{C_A}{\pip(2^m)}r\big)}:=\vint_{B_d\big(x,\tfrac{C_A}{\pip(2^m)}r\big)}u\, d\mu_\pip
		\approx \vint_{B_d\big(x,\tfrac{C_A}{\pip(2^m)}r\big)}u\, d\mu=:c_u.
		\]
		In what follows, by $g_{u,d}$ we mean the minimal $p$-weak upper gradient of $u$ with respect to the metric
		$d$, while $g_{u,\pip}$ denotes the minimal $p$-weak upper gradient with respect to $d_\pip$.
		Hence, by the sub-Whitney Poincar\'e inequality for $(\Omega,d,\mu)$, we have
		\[
		\begin{split}
			\vint_{B_i} \vert u-c_u\vert d\mu_\pip    
			&\lesssim \vint_{B_d(x,\tfrac{C_A}{\pip(2^m)}r)} \vert u-c_u\vert d\mu\\
			& \lesssim   \text{diam}_d(B_d(x,\tfrac{C_A}{\pip(2^m)}r))\left(\vint_{B_d(x,\tfrac{C_A}{\pip(2^m)}r)} g_{u,d}^p d\mu\right)^{1/p}\\
			&\lesssim r_i\left(\vint_{ C_A^2B_i} g_{u,\pip}^p d\mu_\pip\right)^{1/p}.
		\end{split}
		\]
		Here we use the estimate
		\[
		\text{diam}_d(B_d(x,\tfrac{C_A}{\pip(2^m)}r))
		\approx \pip(2^m)^{-1} \text{diam}_{\pip}(B_d(x,\tfrac{C_A}{\pip(2^m)}r))\approx \pip(2^m)^{-1}\, r_i
		\]
		from Lemma~\ref{lem:nearby-points}, together with 
		\[
		g_{u,d}\approx\pip(2^m)g_{u,\pip}
		\]
		to justify the last step. 
		It follows that $(\Om_\pip\setminus\{\infty\}, d_\pip,\mu_\pip)$ satisfies a sub-Whitney $p$-Poincar\'e inequality, i.e.,
		with respect to the balls $B_i$. 
		See Definition~\ref{def:poin} above for these concepts.
		Recall that $\Om_\pip\setminus\{\infty\}$ is a uniform domain, see Theorem~\ref{thm:uniform} above.
		Now we invoke~\cite[Theorem~4.4]{BS} to conclude that 
		$(\Om_\pip\setminus\{\infty\}, d_\pip,\mu_\pip)$ satisfies a $p$-Poincar\'e inequality
		with respect to all balls. While the statement of~\cite[Theorem~4.4]{BS} requires that $p$-Poincar\'e inequality
		be valid with respect to all balls in an ambient space containing the uniform domain, \emph{the proof there only needed
			the validity of $p$-Poincar\'e inequality with respect to the balls in the chain.}
		
		Now the remaining claims follow from~\cite[Proposition~7.1]{AS}, for we have that 
		$\Om_\pip\setminus\{\infty\}\subset\Om_\pip\subset\overline{\Om_\pip}^\pip
		=\overline{\Om_\pip\setminus\{\infty\}}^\pip$.
	\end{proof}

	\section{Transformation of potentials}\label{Sect:Dirichlet}
	
	In this section, we return to the original motivation for the problems studied in the prior sections of this paper.
	We assume that $\Om$ is a unbounded locally compact, non-complete uniform domain with 
	bounded boundary, equipped with a doubling measure $\mu$ that supports a sub-Whitney $p$-Poincar\'e inequality 
	for some fixed $1\leq p <\infty$.
	Here, of course, we extend the measure $\mu_\pip$ to $\partial(\Om_\pip\setminus\{\infty\})$ by zero.
	
	Recall that as a set, $\Om=\Om_\pip\setminus\{\infty\}$. From the results in the prior sections,
	we know that $(\Om_\pip\setminus\{\infty\}, d_\pip,\mu_\pip)$  is doubling and supports a $p$-Poincar\'e 
	inequality; hence by~\cite{AS} we know that 
	\begin{equation}\label{eq:N=N}
		N^{1,p}(\Om_\pip\setminus\{\infty\},d_\pip,\mu_\pip)=N^{1,p}(\Om_\pip,d_\pip,\mu_\pip)=N^{1,p}(\overline{\Om_\pip}^\pip,d_\pip,\mu_\pip).
	\end{equation}
	With $g_{u,d}$ the minimal $p$-weak upper gradient of a function $u\in N^{1,p}(\Om,d,\mu)$ with respect to
	the original metric $d$, and
	$g_{u,\pip}$ the minimal $p$-weak upper gradient of $u$ with respect to the metric $d_\pip$, we have the
	relationship
	\begin{equation}\label{eq:chain-rule}
		g_{u,\pip}=\frac{1}{\pip\circ d_\Om}\, g_{u,d}.
	\end{equation}
	It follows that 
	\begin{equation}\label{eq:chain-rule-integral}
		\int_{\Om_\pip}g_{u,\pip}^p\, d\mu_\pip=\int_\Om g_{u,d}^p\, d\mu.
	\end{equation}
	As a consequence of~\eqref{eq:chain-rule}, we have the following proposition (see Definition~\ref{def:harm}
	for the definition of $p$-harmonicity).
	
	\begin{proposition}\label{prop:harm-harm}
		A function $u$ is $p$-harmonic in the metric measure space $(\Om,d,\mu)$ if and only if it is $p$-harmonic
		in $(\Om_\pip\setminus\{\infty\}, d_\pip,\mu_\pip)$.
	\end{proposition}
	
	Also from~\eqref{eq:chain-rule} we obtain the following proposition.
	
	\begin{proposition}\label{lem:D-to-N}
		Let $u\in D^{1,p}(\Om, d, \mu)$. Then $u\in L^p(\Om,\mu_\pip)$ with 
		\[
		\int_\Om |u-c_u|^p\, (\pip\circ d_\Om)^p\, d\mu\le C\, \int_\Om g_{u,d}^p\, d\mu.
		\]
		Here $c_u=\vint_\Om u\, d\mu_\pip$. In particular, 
		$u\in N^{1,p}(\Om_\pip, d_\pip, \mu_\pip)$.
	\end{proposition}
	
	As $\Om$ is unbounded with respect to the metric $d$, we cannot conclude that $u\in L^p(\Om,\mu)$.
	Note that with the choice of $\pip(t)=\min\{1,t^{-\beta}\}$ for sufficiently large fixed $\beta>1$, the above lemma is an analog
	of a Hardy-Sobolev inequality with distance to $\partial_\pip\overline{\Om}^d=\{\infty\}$ 
	playing the role of distance to the boundary. Readers interested in the topic of Hardy-Sobolev spaces
	are referred to~\cite[Section~1.3.3]{Maz}, \cite{KL} (for the Euclidean setting), and~\cite[Corollary~6.1]{BMS}
	(for a metric setting) and the references therein.
	
	\begin{proof}
		Since $u\in D^{1,p}(\Om,d,\mu)$, we know that the minimal $p$-weak upper gradient $g_{u,d}$ of $u$ in 
		$\Om$, with respect to the 
		metric $d$, is in $L^p(\Om,\mu)$. Then by~\eqref{eq:chain-rule} above, we know that 
		$u\in D^{1,p}(\Om_\pip\setminus\{\infty\},d_\pip,\mu_\pip)$, recalling that $\Om=\Om_\pip\setminus\{\infty\}$. Note that then $u\in N^{1,p}(\Om_\pip\setminus\{\infty\},d_\pip,\mu_\pip)$ by
		Remark~\ref{D=P}.
		From Theorem~\ref{thm:PI-no-infty} and 
		Theorem~\ref{thm:double}, 
		we know that $\Om_\pip\setminus\{\infty\}$ is bounded with respect to the metric $d_\pip$, and supports a $p$-Poincar\'e inequality with respect
		to $d_\pip$ and $\mu_\pip$. 
		Thus, for any $x\in\Om_\pip\setminus\{\infty\}$ and for sufficiently large $R>0$, we have $\Om_\pip\setminus\{\infty\}=B_{d_\pip}(x,R)$, and moreover, 
		$(\Om_\pip\setminus\{\infty\},d_\pip,\mu_\pip)$ also supports the following $(p,p)$-Poincar\'e inequality, see~\cite[Theorem~5.1]{HaK},
		\cite[Theorem~9.1.2]{HKST}:
		for $u\in N^{1,p}(\Om_\pip\setminus\{\infty\},d_\pip,\mu_\pip)$, we have by~\eqref{eq:chain-rule-integral} that
		\[
		\int_{\Om_\pip\setminus\{\infty\}} |u-c_u|^p\, d\mu_\pip\le C\, \int_{\Om_\pip\setminus\{\infty\}} g_{u,\pip}^p\, d\mu_\pip= C\, \int_\Om g_{u,d}^p\, d\mu.
		\]
		Note that $C$ also depends on $R$. Now the desired conclusion follows from noting that, by definition, $d\mu_\pip=(\pip\circ d_\Om)^p\, d\mu$ and by
		using~\eqref{eq:N=N}.
	\end{proof}

	As $\mu$ is doubling, there exists some $Q^-_\mu>0$, called the {\em lower mass bound exponent} of $\mu$, such that
	\[
	\frac{ \mu(B_d(x,r)) }{ \mu(B_d(y,R)) }\gtrsim \left(\frac{r}{R}\right)^{Q^-_\mu}
	\]
	for all $x,y\in\Omega$ with $x\in B_d(y,R)$ and $0<r\leq R<\infty$, where the implied constant depends only on $C_\mu$, the doubling constant of $\mu$, see for instance~\cite[(4.16)]{Hei} or \cite[Lemma~8.1.13]{HKST}.  
	Moreover, as $\Om$ is connected, there exists $Q^+_\mu>0$, called the {\em upper mass bound exponent} of $\mu$, such that
	\[
	\frac{ \mu(B_d(x,r)) }{ \mu(B_d(y,R)) }\lesssim \left(\frac{r}{R}\right)^{Q^+_\mu}
	\]
	for all $x,y\in\Omega$ with $x\in B_d(y,R)$ and $0<r\leq R<\infty$, see~\cite[Corollary~3.8]{BB}. Note that $Q_\mu^+\le Q_\mu^-$.

	\begin{lemma}\label{lem:lower-mass}
		Let $1\le p<\infty$, and
		$\pip(t)=\min\{1,t^{-\beta}\}$ for $t>0$ and $\beta>1$ such that $\beta p>Q_\mu^-$.
		For all $0<r\leq R<r_0$, where $r_0$ is as in Section~\ref{sec:doubling}, 
		$$
		\left(\frac{r}{R}\right)^{Q^+_\beta} \gtrsim
		\frac{ \mu_\pip(B_{d_\pip}(\infty,r)) }{ \mu_\pip(B_{d_\pip}(\infty,R)) }\gtrsim \left(\frac{r}{R}\right)^{Q^-_\beta},
		$$  
		where $Q^-_\beta=\frac{\beta p - Q^+_\mu}{\beta-1}$ and $Q^+_\beta=\frac{\beta p - Q^-_\mu}{\beta-1}$. 
		Moreover, the function $\pip$ satisfies the conditions of Definition~\ref{def:phi}.
	\end{lemma}
	
	\begin{proof}
		Fix $0<r<R<r_0$. Take non-negative integers $m_r$ and $m_R$ such that $\Omega_{m_r}$ contains a 
		point $x_r$ satisfying $d_\pip(x_r,\infty)=r$ and $\Omega_{m_R}$ contains a point $x_R$ satisfying 
		$d_\pip(x_R,\infty)=R$. By Lemma~\ref{lem:balls-at-infty} and the assumption \eqref{additionalmupip},
		\begin{align*}
			\frac{\mu_\pip(B_{d_\pip}(\infty,r))}{\mu_\pip(B_{d_\pip}(\infty,R))}
			\approx \frac{ \sum_{n=m_r}^{\infty}\pip(2^n)^p\mu(\Om_n) }{\sum_{n=m_R}^{\infty}\pip(2^n)^p\mu(\Om_n)}	
			&\approx
			\left(\frac{\pip(2^{m_r})}{\pip(2^{m_R})}\right)^p\frac{\mu(\Om_{m_r})}{\mu(\Om_{m_R})}\\
			&= \left(\frac{2^{m_R\beta}}{2^{m_r\beta}}\right)^p\frac{\mu(\Om_{m_r})}{\mu(\Om_{m_R})}. 
		\end{align*}
		
		From Lemma~\ref{lem:comparable-layers}, there exist $y_r\in\Om_{m_r}$ with $d_\Om(y_r)=2^{m_r}$ and $y_R\in\Om_{m_R}$ with $d_\Om(y_R)=2^{m_R}$, so that
		\[
		\frac{\mu(\Om_{m_r})}{\mu(\Om_{m_R})}\approx \frac{ \mu(B_d(y_r,2^{m_r})) }{\mu(B_d(y_R,2^{m_R}))}.
		\]
		As $\partial\Om$ is bounded with respect to $d$, we can consider an integer upper bound $K$ for $\diam_d(\partial\Om)$. The ball $B_d(y_r,2K2^{m_r})$ engulfs $\partial\Om$ and therefore also $B_d(y_R,2^{m_R})$ since from the above estimates
		we can conclude that $2^{m_{R}}\lesssim 2^{m_{r}}$. Hence, by the doubling property of $\mu$,
		\[
		\frac{ \mu(B_d(y_r,2^{m_r})) }{\mu(B_d(y_R,2^{m_R}))}\approx\frac{ \mu(B_d(y_r,2K2^{m_r})) }{\mu(B_d(y_R,2^{m_R}))}.
		\]
		Applying the upper and lower mass bound estimates for $\mu$ to this right-hand quantity, we arrive at
		\[
		\left(\frac{2^{m_r}}{2^{m_R}}\right)^{Q^+_\mu}\lesssim \frac{ \mu(B_d(y_r,2^{m_r})) }{\mu(B_d(y_R,2^{m_R}))}\lesssim \left(\frac{2^{m_r}}{2^{m_R}}\right)^{Q^-_\mu}.
		\]
		By Lemma~\ref{lem:dist-to-infty}, we have that $r=d_\pip(x_r,\infty)\approx 2^{m_r}\pip(2^{m_r})=2^{m_r(1-\beta)}$ and similarly for $R$. Therefore, 
		\[
		\frac{\mu_\pip(B_{d_\pip}(\infty,r))}{\mu_\pip(B_{d_\pip}(\infty,R))}
		\gtrsim\left(\frac{2^{m_R}}{2^{m_r}}\right)^{\beta p-Q^+_\mu} \approx \left( \frac{R^{\frac{1}{1-\beta}}}{r^{\frac{1}{1-\beta}}} \right)^{\beta p-Q^+_\mu}=\left( \frac{r}{R} \right)^{\frac{\beta p-Q^+_\mu}{\beta-1}},
		\]
		with the opposite relationships holding with $Q_\mu^+$ replaced by $Q_\mu^-$.
		
		The first five conditions of Definition~\ref{def:phi} are clear for this choice of $\pip$. Condition~(6) of that definition
		follows from Lemma~\ref{lem:comparable-layers}. Indeed, in this case, for $n,m\in\N$ with $n>m>n_0$, we have the existence
		of points $y_n\in\Om_n$ and $y_m\in\Om_m$ such that
		$\mu(\Om_m)\approx\mu(B_d(y_m,2^m))$ and $\mu(\Om_n)\approx\mu(B_d(y_n,2^n))$. Moreover, $B_d(y_m,2^m)$ intersects
		$B_d(y_n,2^n)$, and so we have that
		\[
		\frac{\mu(\Om_n)}{\mu(\Om_m)}\lesssim\frac{\mu(B_d(y_n,2^n))}{\mu(B_d(y_m,2^m))}\lesssim \left(\frac{2^n}{2^m}\right)^{Q_\mu^-},
		\]
		and so
		\begin{align*}
			\sum_{n=m}^\infty\pip(2^n)^p\, \mu(\Om_n)&\lesssim \sum_{n=m}^\infty 2^{-\beta pn}\, 2^{(n-m)Q_\mu^-}\, \mu(\Om_m)\\
			&=2^{-mQ_\mu^-}\, \mu(\Om_m)\, \sum_{n=m}^\infty 2^{n(Q_\mu^--\beta p)}\\
			&=2^{-\beta p\, m}\, \mu(\Om_m)\, \sum_{j=0}^\infty 2^{j(Q_\mu^--\beta p)}.
		\end{align*}
		Since $\beta p>Q_\mu^-$, the latter series converges, and hence Condition~(6) follows.
	\end{proof}

	\begin{proposition}\label{thm:bdry-mod}
		Let $1\le p<\infty$, 
		$\pip(t)=\min\{1,t^{-\beta}\}$ for $t>0$ and $\beta>1$ such that $\beta p>Q_\mu^-$,
		and let $Q_\beta^+, Q_\beta^-$ be as in Lemma~\ref{lem:lower-mass}.
		Let $\Gamma$ be the collection of all curves in $\Om_\pip$ that are non-constant, compact, and
		rectifiable with respect to the metric $d_\pip$, and ending at $\infty$.
		\begin{enumerate}
			\item If $p> Q_\beta^-$, then $\Mod_p^\pip(\Gamma)>0$.
			\item If $p< Q_\beta^+$ or $1<p=Q_\beta^+$, then $\Mod_p^\pip(\Gamma)=0$.
		\end{enumerate}
	\end{proposition} 
	
	Note that $p>Q^-_\beta$ if and only if $p<Q^+_\mu$, and $p<Q_\beta^+$ if and only if $p>Q_\mu^-$.
	
	\begin{proof}
		We first prove (1). Fix $0<r<R<r_0$, and choose a positive integer $k_0$ such that
		$\mu_\pip(B_{d_\pip}(\infty,R))/\mu_\pip(B_{d_\pip}(\infty,2^{k_0-1}R))\le 1/2$.
		We begin by showing that 
		$\Mod_p(\overline{B}_{d_\pip}(\infty,r); \Om_\pip\setminus B_{d_\pip}(x_0,R))\ge CR^{Q_\beta^--p}$
		for some constant $C>0$ that is independent of $r$. Since the $p$-modulus of the condenser is equal to the variational
		$p$-capacity of the condenser with $U=\Om_\pip$
		(see Remark~\ref{rem:capmod}), we will work with the latter and utilize the Poincar\'e inequality. 
		
		To this end,
		consider a function $u\in N^{1,p}(\Om_\pip)$ satisfying $u=1$ on $\overline{B}_{d_\pip}(\infty,r)$ and $u=0$ in 
		$\Om_\pip\setminus B_{d_\pip}(\infty, R)$.  
		Then, $\infty$ is a Lebesgue point of $u$, and so for each positive integer $k$ setting $B_k=B_{d_\pip}(\infty, 2^{k_0-k}R)$, we have that
		\begin{align*}
			1-u_{B_1}=|u(\infty)-u_{B_1}|&\le \sum_{k=1}^\infty |u_{B_k}-u_{B_{k+1}}|\\
			&\lesssim \sum_{k=1}^\infty 2^{k_0-k}R\, \left(\vint_{B_k}g_{u,\pip}^p\, d\mu_\pip\right)^{1/p}\\
			&\lesssim \mu_\pip(B_1)^{-1/p}\sum_{k=1}^\infty (2^{k_0-k}R)^{1-\tfrac{Q_\beta^-}{p}}\left(\int_{B_k}g_{u,\pip}^p\, d\mu_\pip\right)^{1/p},
		\end{align*}
		where we used Lemma~\ref{lem:lower-mass} in the last step. From the choice of $k_0$,
		\[
		u_{B_1}=\frac{1}{\mu_\pip(B_{d_\pip}(\infty,2^{k_0-1}R))}\int_{B_{k_0}}u\, d\mu_\pip
		\le \frac{\mu_\pip(B_{d_\pip}(\infty,R))}{\mu_\pip(B_{d_\pip}(\infty,2^{k_0-1}R))}\le \frac12.
		\]
		
		Since $p>Q_\beta^-$, from the above we obtain
		\[
		\frac12\le 1-u_{B_1}\lesssim \mu_\pip(B_1)^{-1/p}R^{1-\tfrac{Q_\beta^-}{p}}\, \left(\int_{\Om_\pip} \!g_{u,\pip}^p\, d\mu_\pip\right)^{1/p},
		\]
		from where it follows that
		\[
		2^{-p}\,\mu_\pip(B_1)\, R^{Q_\beta^--p}\lesssim \int_{\Om_\pip}\!g_{u,\pip}^p\, d\mu_\pip,
		\]
		and, taking the infimum over all such $u$,
		\[
		\rcapa_p(\overline{B}_{d_\pip}(\infty,r), \Om_\pip\setminus B_{d_\pip}(x_0,R))\gtrsim \mu_\pip(B_1)\,R^{Q_\beta^--p}>0.
		\]
		Letting $r\to 0^+$, from the Choquet property of variational $p$-capacity (see \cite[Theorem~6.7(viii)]{BB}) we obtain
		\[
		\rcapa_p(\{\infty\}, \Om_\pip\setminus B_{d_\pip}(x_0,R))\gtrsim \mu_\pip(B_1)\,R^{Q_\beta^--p}>0.
		\]
		The result then follows from Remark~\ref{rem:capmod}. 
		
		We now move to proving (2). Fix $0<r<\frac{R}{4}<R<r_0$, and let $n_r$ be the unique
		positive integer such that $2^{-n_r}R\le r<2^{1-n_r}R$. Let $\rho$ be the function on $\Om_\pip$ given by
		$\rho(x)=2\,[n_rd_\pip(x,\infty)]^{-1}\chi_{B_{d_\pip}(\infty, R)\setminus B_{d_\pip}(\infty, r)}$. 
		For each $\gamma\in \Gamma(\overline{B}_{d_\pip}(\infty, r), \Om_\pip\setminus B_{d_\pip}(\infty, R),\Om_\pip)$,
		we have $\int_\gamma\rho\, ds_\pip\ge 1$. 
		
		Setting $B_k=B_{d_\pip}(\infty,2^{1-k}R)$ for $k=1,\cdots, n_r$, we see that
		$$
		\int_{\Om_\pip}\!\rho^p\, d\mu_\pip \leq \sum_{k=1}^{n_r}\int_{B_k\setminus B_{k+1}}\!\frac{2^p}{n_r^{p}d_\pip(x,\infty)^p}\,d\mu_\pip(x)\leq\sum_{k=1}^{n_r}\frac{2^{(1+k)p}}{n_r^{p}R^p}\,\mu_\pip(B_k\setminus B_{k+1})
		$$
		since $x\notin B_{k+1}$ implies that $d_\pip(x,\infty)\geq 2^{-k}R$.
		We now estimate $\mu_\pip(B_k\setminus B_{k+1})$. From Lemma~\ref{lem:dist-to-infty}, it follows that $B_k\setminus B_{k+1}\subset\Om_{m_k}$ for some non-negative integer $m_k$; moreover, $2^{-k}R\approx d_\pip(x,\infty)\approx 2^{m_k(1-\beta)}$ for $x\in B_k\setminus B_{k+1}$. Thus, for some $y_k\in \Om_{m_k}$ (see Lemma~\ref{lem:comparable-layers}),
		\begin{align*}
			\mu_\pip(B_k\setminus B_{k+1})&=\int_{B_k\setminus B_{k+1}}\!d_\Om(x)^{-p\beta}\,d\mu\approx 2^{-m_kp\beta}\mu(B_k\setminus B_{k+1})\\&\leq2^{-m_kp\beta}\mu(\Om_{m_k}) \leq 2^{-m_kp\beta}\mu(B_d(y_k,2^{m_k} ))\\ 
			&\lesssim 2^{-m_kp\beta}(2^{m_k})^{Q_\mu^-}
		\end{align*} 
		where the lower-mass bound on $\mu$ was used in the last step.
		
		From here, we have 
		\begin{equation}\label{eq:last-step}
			\int_{\Om_\pip}\!\rho^p\, d\mu_\pip\lesssim \sum_{k=1}^{n_r}\frac{2^{kp}}{n_r^{p}R^p}\,2^{m_k(Q_\mu^--p\beta)} \approx \frac{R^{Q_\beta^+-p}}{n_r^p}\sum_{k=1}^{n_r}(2^k)^{p-Q_\beta^+}. 
		\end{equation}
		Recall that $Q_\beta^+=\frac{p\beta-Q_\mu^-}{\beta-1}$ and $2^{m_k}\approx (2^{-k}R)^{\frac{1}{1-\beta}}$. 
		
		If $1<p=Q_\beta^+$, then the sum on the right-hand side of \eqref{eq:last-step} equals $n_r$, and so
		\[
		\Mod_p^\pip(\overline{B}_{d_\pip}(\infty,r); \Om_\pip\setminus B_{d_\pip}(x_0,R))\lesssim n_r^{1-p}\approx \left[\log\left(\frac{R}{r}\right)\right]^{1-p},
		\]
		the right side of which tends to zero as $r\rightarrow{0^+}$.
		If $p<Q_\beta^+$, then the sum on the right-hand side of \eqref{eq:last-step} is dominated by the convergent 
		series obtained by summing over all positive integers $k$, and so
		\[
		\Mod_p^\pip(\overline{B}_{d_\pip}(\infty,r); \Om_\pip\setminus B_{d_\pip}(x_0,R))\lesssim n_r^{-p}\approx \left[\log\left(\frac{R}{r}\right)\right]^{-p},
		\]
		the right side of which also tends to zero as $r\rightarrow{0^+}$.
		
		In either case, it then follows that $\Mod_p^\pip(\{\infty\};\Om_\pip\setminus B_{d_\pip}(\infty,R))=0$.
		From this it follows that the $p$-modulus 
		(with respect to $d_\pip$ and $\mu_\pip$) of the collection of all non-constant compact rectifiable curves in $\Om_\pip$
		that intersect $\infty$ is zero as well.
	\end{proof}

	\begin{remark}\label{rem:7.8}
		As a consequence of the above results and by~\cite{Bjo}, we have that if $u$ is $p$-harmonic in $(\Om,d,\mu)$
		and $p>Q_\mu^-$ or $p=Q_\mu^->1$, 
		then $u$ is $p$-harmonic in $(\Om_\pip,d_\pip,\mu_\pip)$. From~\cite[Proposition~7.2]{Bjo}
		it also follows that when $p<Q_\mu^+$ and the Sobolev $p$-capacity 
		$\text{Cap}_p(\partial\Om)$ is positive, then there is a bounded $p$-harmonic function 
		in $(\Om_\pip\setminus\{\infty\}, d_\pip,\mu_\pip)$ which has no $p$-harmonic extension to $\Om_\pip$.
	\end{remark}
	
	\section{Connecting potential theory on $\Om$ to functions on $\partial\Om$}\label{Sec:eight}
	
	Let $\nu$ be a Borel regular measure on $\partial\Om=\partial\Om_\pip$ and $\theta\geq 0$. We say that $\nu$ is 
	\emph{$\theta$-codimensional to} $\mu_\pip$ if 
	\[
	\nu(B_{d_\pip}(\zeta,r)\cap\partial\Om)\approx \frac{\mu_\pip(B_{d_\pip}(\zeta,r)\cap\Om_\pip)}{r^\theta}
	\]
	for all $\zeta\in\partial\Om$ and $0<r<2\diam_{d_\pip}(\partial\Om)$.
	Note that as $\mu_\pip$ is doubling on $\Om_\pip\setminus\{\infty\}$, we must necessarily have that 
	$\nu$ is doubling on $\partial\Om$ with respect to the metric $d_\pip$ (and equivalently, with respect to the metric $d$).
	
	\begin{lemma}\label{lem:Haus}
		Assume there exists a Borel regular measure $\nu$ on $\partial\Om$ that is $\theta$-codimensional to 
		$\mu$ for some $\theta\ge 0$. Then, $\mathcal{H}^{-\theta}(A;\overline{\Om})\approx \nu(A)$ for all $A\subset\partial\Om$.
		
		If $\theta>0$, then $\widehat{\mu}_\pip(\partial\Om)=0$ for any doubling measure $\widehat{\mu}_\pip$ on $\overline{\Om}_\pip^\pip$
		such that $\widehat{\mu}_\pip=\mu_\pip$ on $\Om_\pip$ and 
		$\nu(B_{d_\pip}(\zeta,r)\cap\partial\Om)\approx \tfrac{\widehat{\mu}_\pip(B_{d_\pip}(\zeta,r))}{r^\theta}$
		for all $\zeta\in\partial\Om$ and $0<r<2\diam_{d_\pip}(\partial\Om)$.
	\end{lemma}
	
	As both $d_\pip$ and $d$ are bi-Lipschitz equivalent on $\partial\Om$ and $\mu=\mu_\pip$ on $\Om_0$, 
	computing the codimensional Hausdorff measure
	$\mathcal{H}^{-\theta}$ of a subset of $\partial\Om$ (as in Definition~\ref{def:codim-Haus} with
	$U=\overline{\Om}$) with respect to the metric $d$ and the measure 
	$\mu$ is equivalent to computing it
	with respect to $d_\pip$ and $\mu_\pip$. In doing the computation of $\mathcal{H}^{-\theta}(A; \overline{\Om})$,
	we extend the measures $\mu$ and $\mu_\pip$ by zero to $\partial\Om$ so that $\mu(\partial\Om)=0$.
	
	\begin{proof}
		We fix $A\subset\partial\Om$.
		
		For each $\eps>0$ let $\{B_{d_\pip}(x_i, r_i)\}$ be a cover of $A$ with $r_i\le \eps$
		and $x_i\in A$.
		Then
		\[
		\sum_i\frac{\mu_\pip(B_{d_\pip}(x_i,r_i))}{r_i^\theta}\gtrsim \sum_i \nu( B_{d_\pip}(x_i, r_i)\cap\partial\Om )\geq\nu(A). 
		\]
		Taking the infimum over all such covers, and then letting
		$\eps\to 0^+$, this implies that $\mathcal{H}^{-\theta}(A;{\overline{\Om_\pip}})\gtrsim\nu(A)$.
		
		Since $\nu$ is Borel regular, it follows that for each $\eta>0$ there is a set $U\supset A$ such that $U$ is open in
		$\overline{\Om}$ and
		$\nu(U)\le \nu(A)+\eta$. Fix $\eps>0$,
		and consider a cover $\{B_{d_\pip}(x_i,r_i)\}$ of $A$ with $r_i\le \eps/5$, $x_i\in A$, 
		and $B_{d_\pip}(x_i,5r_i)\subset U$. By the basic $5r$-covering lemma 
		as in~\cite[Theorem~1.2]{Hei},  
		there is a countable pairwise disjoint subcollection $\{B_{d_\pip}(x_j,r_j)\}$ such that $\{B_{d_\pip}(x_j,5r_j)\}$ covers $A$. Then,
		by the doubling property of $\nu$,
		\begin{align*}
			\mathcal{H}^{-\theta}_\eps(A;\overline{\Om_\pip})\le \sum_j\frac{\mu_\pip(B_{d_\pip}(x_j,5r_j))}{(5r_j)^\theta}
			&\lesssim \sum_{j}\frac{\mu_\pip(B_{d_\pip}(x_j,r_j))}{r_j^\theta}\\
			&\lesssim \sum_j \nu( B_{d_\pip}(x_j,r_j)\cap\partial\Om )\\
			&=\nu( \bigcup _j B_{d_\pip}(x_j,r_j)\cap\partial\Om)\\
			&\leq \nu(U\cap\partial\Om)\ \le\  \nu(A)+\eta.
		\end{align*} 
		Letting $\eps\to 0^+$,
		we obtain $\mathcal{H}^{-\theta}(A;{\overline{\Om_\pip}})\lesssim\nu(A)+\eta$. Letting $\eta\to 0^+$ now yields the desired conclusion
		of the first part of the lemma. 
		
		To prove the second claim 
		of the lemma, we suppose that $\theta>0$ and
		argue as in the above proof to obtain that
		for each $\eps>0$ and any cover $\{B_{d_\pip}(x_i,r_i)\}$ of $\partial\Om$ with $x_i\in\partial\Om$ and $r_i\le \eps$, to see that
		\[
		\frac{\widehat{\mu}_\pip(\partial\Om)}{\eps^\theta}
		\le \frac{\sum_i\widehat{\mu}_\pip(B_{d_\pip}(x_i,r_i))}{\eps^\theta}\le \sum_i\frac{\mu_\pip(B_{d_\pip}(x_i,r_i))}{r_i^\theta}
		\lesssim \nu(\partial\Om)+\eta,
		\]
		from which we obtain that $\widehat{\mu}_\pip(\partial\Om)\lesssim \eps^\theta\, [\nu(\partial\Om)+\eta]$. Letting $\eps\to 0^+$
		yields the conclusion as $\nu(\partial\Om)$ is finite (note that $\partial\Om$ is bounded).
	\end{proof}
	
	\begin{proposition}\label{prop:bdryCap}
		Let $1\leq p <\infty$, and assume there exists a Borel regular measure $\nu$ on $\partial\Om$ that is $\theta$-codimensional to 
		$\mu_\pip$ for $0< \theta<p$. Then, for all $\zeta\in\partial\Om$ and $r>0$, we have
		$\capa_p^\pip(B_{d_\pip}(\zeta,r)\cap\partial\Om)>0$.
	\end{proposition}
	
	\begin{proof}
		From Theorem~\ref{thm:double} and Theorem~\ref{thm:PI-no-infty} we have that $\overline{\Om}^\pip$, equipped with the 
		metric $d_\pip$ and the
		measure $\mu_\pip$, is doubling and supports a $p$-Poincar\'e inequality. Note that $\mu=\mu_\pip$ on
		$\Om_0$ and $d_\pip$ is bi-Lipschitz equivalent to $d$ on $\Om_0\cup\partial\Om$. 
		
		Due to Proposition~\ref{prop:cap-haus}, it suffices to show that 
		$\mathcal{H}^{-\theta}(B_{d_\pip}(\zeta,r)\cap\partial\Om)>0$. From Lemma~\ref{lem:Haus}, 
		$\mathcal{H}^{-\theta}(B_{d_\pip}(\zeta,r)\cap\partial\Om)\approx \nu(B_{d_\pip}(\zeta,r)\cap\partial\Om)$, which must be positive 
		because $\nu$ is doubling. 
	\end{proof}

	Recall that the space $D^{1,p}(\Om,d,\mu)$ consists of measurable 
	functions on $\Om$ which have an upper gradient in $L^p(\Om,\mu)$. This space is naturally equipped 
	with the seminorm $\Vert u\Vert_{D^{1,p}(\Om,d,\mu)}:=\Vert g_{u,d}\Vert_{L^p(\Om,\mu)}$;
	note from~\eqref{eq:chain-rule} that 
	$\Vert g_{u,d}\Vert_{L^p(\Om,\mu)}=\Vert g_{u,\pip}\Vert_{L^p(\Om_\pip,\mu_\pip)}$.
	
	\begin{proposition}\label{prop:Dp-trace}
		Let $1\leq p <\infty$, and assume there exists a Borel regular measure $\nu$ on $\partial\Om$ that is $\theta$-codimensional to $\mu$ for 
		$0< \theta<p$. For every $u\in D^{1,p}(\Om,d,\mu)$ there exists $Tu\in B^{1-\theta/p}_{p,p}(\partial\Om,\nu)$ such that
		\begin{equation}\label{eq:trace}
			\lim_{r\rightarrow{0^+}}\vint_{B_d(\zeta,r)\cap\Om}\!|u(x)-Tu(\zeta)|\,d\mu(x)=0
		\end{equation}
		for $\nu$-a.e. $\zeta\in\partial\Om$. Moreover, the operator $u\mapsto Tu$ is bounded from $D^{1,p}(\Om,d,\mu)$ to 
		$B^{1-\theta/p}_{p,p}(\partial\Om,\nu)$.
	\end{proposition}
	
	Recall that since $\partial\Om$ is bounded, $B^{1-\theta/p}_{p,p}(\partial\Om,\nu)\subset L^p(\partial\Om,\nu)$. 
	
	\begin{proof}
		By Lemma~\ref{lem:D-to-N} and~\eqref{eq:N=N}, 
		it suffices to look at $u\in N^{1,p}(\Om_\pip,d_\pip,\mu_\pip)$. Since $\Om_\pip$ is a uniform domain with compact closure (see Proposition~\ref{prop:compactness}), it follows from Proposition~\ref{prop:Besov-trace} that there exists $Tu\in B^{1-\theta/p}_{p,p}(\partial\Om,\nu)$ satisfying
		\begin{equation}\label{eq:pip-trace}
			\lim_{r\rightarrow{0^+}}\vint_{B_{d_\pip}(\zeta,r)\cap\Om_\pip}\!|u(x)-Tu(\zeta)|\,d\mu_\pip(x)=0
		\end{equation}
		for $\nu$-a.e. $\zeta\in\partial\Om$ and 
		$\|Tu\|_{B^{1-\theta/p}_{p,p}(\partial\Om,\nu)}\lesssim \|u\|_{D^{1,p}(\Om_\pip,d_\pip,\mu_\pip)}=\|u\|_{D^{1,p}(\Om,d,\mu)}$. 
		
		For $\zeta\in\partial\Om$ and for sufficiently small $r>0$, we have that $\mu_\pip=\mu$ on $B_d(\zeta,r)$ which is
		bi-Lipschitz equivalent to $B_{d_\pip}(\zeta,r)$. Therefore, \eqref{eq:pip-trace} is equivalent to \eqref{eq:trace}. \qedhere
	\end{proof}

	We now turn our attention to proving an Adams-type inequality on $\overline{\Om_\pip}^\pip$ with respect to the
	measure $\nu$ supported on $\partial\Om$.

	\begin{remark}\label{rem:better-PI}
		From Theorems~\ref{thm:double} and~\ref{thm:PI-no-infty}, it follows that $(\overline{\Om_\pip}^\pip, d_\pip,\mu_\pip)$ 
		is doubling and supports
		a $p$-Poincar\'e inequality. Hence by~\cite{KZ} (see also~\cite[Theorem~12.3.9]{HKST}), 
		$(\overline{\Om_\pip}^\pip, d_\pip,\mu_\pip)$ also supports a $\widetilde{p}$-Poincar\'e inequality for some $1\le \tilde{p}<p$.
	\end{remark} 
	
	In the following, $\Qpip$ is the lower mass bound for the measure $\mu_\pip$. 
	
	\begin{theorem}\label{thm:Adams}
		Let $1<p<\infty$, and assume there exists a Borel regular measure $\nu$ on $\partial\Omega$ that is 
		$\theta$-codimensional to $\mu_\pip$ with $0< \theta<p<\Qpip$, and let $q>p$ be a real number given by
		\[
		\theta=-\frac{\Qpip\ q}{p} + \Qpip + \frac{q}{\tilde{p}},
		\]
		where $1\leq \widetilde{p}<p$ is from Remark~\ref{rem:better-PI}.  
		Then for every 
		$u\in N^{1,p}(\overline{\Omega_\pip}^\pip,d_\pip,\mu_\pip)$ and ball $B_{d_\pip}\subset \overline{\Omega_\pip}^\pip$,
		\[
		\inf_{c\in\R}\left(\int_{B_{d_\pip}}\!|u-c|^q\,d\nu\right)^{1/q}
		\lesssim \frac{\rad_{d_\pip}(B_{d_\pip})^{1-\theta/q}}
		{\mu_\pip(B_{d_\pip})^{1/p-1/q}} 
		\left(\int_{2B_{d_\pip}}\!g_{u,\pip}^p\,d\mu_\pip\right)^{1/p}.
		\]
		In particular,
		\[
		\inf_{c\in\R}\left(\int_{\partial\Om}\!|u-c|^q\,d\nu\right)^{1/q}
		\lesssim \frac{\mathrm{diam}_{d_\pip}(\Omega_\pip)^{1-\theta/q}}
		{\mu_\pip(\Omega_\pip)^{1/p-1/q}}
		\left(\int_{\Om_\pip}\!g_{u,\pip}^p\,d\mu_\pip\right)^{1/p}.
		\]
	\end{theorem}
	
	\begin{proof}
		Fix a ball $B_{d_\pip}:=B_{d_\pip}(y,r)\subset \overline{\Omega_\pip}^\pip$ and $u\in \Lip(\overline{\Omega_\pip}^\pip)$. Since $\overline{\Omega_\pip}^\pip$ is a compact length space, it follows that it is a geodesic space, and compactness implies that $u$ is a bounded continuous function. Moreover, it satisfies a $\widetilde{p}$-Poincar\'{e} inequality by Remark~\ref{rem:better-PI}. Hence, setting 
		\[
		u_{2B_{d_\pip}}:=\vint_{2B_{d_\pip}}u\, d\mu_\pip,
		\]
		it follows from~\cite[Theorem~9.5]{Hei} that 
		$$
		|u(x)-u_{2B_{d_\pip}}|^{\widetilde{p}}\lesssim r^{{\widetilde{p}}-1}I_{2B_{d_\pip}}(g_{u,\pip}^{\widetilde{p}})(x)
		$$
		for $x\in B_{d_\pip}$, where 
		$I_{2B_{d_\pip}}$ is the Riesz potential, see Definition~\ref{Riesz}. 
		
		Integrating and using Proposition~\ref{prop:Riesz-neu}
		with $\bar{p}=\frac{p}{\widetilde{p}}$ and $\bar{q}=\frac{q}{\widetilde{p}}$ yields  
		\begin{align*}
			&\left(\int_{B_{d_\pip}}\!|u-u_{2B_{d_\pip}}|^q\,d\nu\right)^{1/q}
			=\left(\int_{B_{d_\pip}}\!(|u-u_{2B_{d_\pip}}|^{\widetilde{p}})^{q/\widetilde{p}}\,d\nu\right)^{1/q}\\
			&\qquad\ \
			\lesssim r^{\frac{\widetilde{p}-1}{\widetilde{p}}}\left(\left(\int_{2B_{d_\pip}}I_{2B_{d_\pip}}(g_{u,\pip}^{\widetilde{p}})^{q/\widetilde{p}}\,d\nu\right)^{\widetilde{p}/q}\right)^{1/\widetilde{p}}\\
			&\qquad\ \
			\lesssim r^{\frac{\widetilde{p}-1}{\widetilde{p}} +\frac{\Qpip}{p}-\frac{\Qpip}{q} }\mu_\pip(2B_{d_\pip})^{1/q-1/p} \left(\left(\int_{2B_{d_\pip}}(g_{u,\pip}^{\widetilde{p}})^{p/\widetilde{p}}\,d\mu_\pip\right)^{\widetilde{p}/p}\right)^{1/\widetilde{p}}\\
			&\qquad\ \
			\lesssim r^{\frac{\widetilde{p}-1}{\widetilde{p}} +\frac{\Qpip}{p}-\frac{\Qpip}{q} }\mu_\pip(B_{d_\pip})^{1/q-1/p} \left(\int_{2B_{d_\pip}}g_{u,\pip}^p\,d\mu_\pip\right)^{1/p}. 
		\end{align*}
		The result follows from the density of the Lipschitz functions in $N^{1,p}(\overline{\Om_\pip}^\pip)$, 
		see, for example, \cite[Theorem~8.2.1]{HKST}. 
	\end{proof}
	
	The above Adams-type inequality also gives us a way to link traces, to $\partial\Om$, of Dirichlet-Sobolev functions on $\Om$,
	and through this, we will see next that the relative capacities of compact subsets of $\partial\Om$ are governed by the
	$\nu$-measure of those subsets.
	As a consequence of the above Adams-type inequality, we have the following corollary giving us a lower bound on relative
	capacities of subsets of $\partial\Om$. 
	
	Given a domain $\Om$ and a point $x\in\partial\Om$, the point $x$ is called
	a \emph{regular point} for the domain if for each continuous function $f:\partial\Om\to\R$ we have
	\[
	\lim_{\Om\ni y\to x}H_\Omega f(y)=f(x),
	\]
	where $H_\Omega f$ is the Perron solution to the Dirichlet problem of finding $p$-harmonic functions on $\Om$ with trace
	$f$. We refer the interested reader to~\cite[Chapter~10--11]{BB}
	or~\cite{BBS-Per} for more on classification of boundary points and Perron solutions.

	\begin{corollary}\label{cor:relcap-lowerbd}
		Let $1<p<\infty$. For each $0<r\le \min\{1,\diam_{d_\pip}(\partial\Om)\}/2$ and $\zeta\in\partial\Om$, we have
		\[
		\nu(B_{d_\pip}(\zeta,r)\cap\partial\Om)\lesssim r^{p-\theta}\, \rcapa_p^\pip(B_{d_\pip}(\zeta,r)\cap\partial\Om,B_{d_\pip}(\zeta,2r)).
		\]
		Consequently, as $0<\theta<p$, then each point of $\partial\Om$ is a regular point for the domain $\Om_\pip$. 
	\end{corollary}
	
	Recall that $\mu_\pip$ is doubling, and $\overline{\Om_\pip}^\pip$ is connected, and so we have a reverse doubling property:
	there is a positive constant $c<1$ such that for all $x\in\overline{\Om_\pip}^\pip$ and 
	$0<r<\tfrac12 \diam_{d_\pip}(\Om_\pip)$, we have $\mu_\pip(B_{d_\pip}(x,r))\le c\, \mu_\pip(B_{d_\pip}(x,2r))$.
	
	\begin{proof}
		Fix $\zeta$ and $r$ as in the statement of the corollary.
		Let $u\in N^{1,p}(\Om_\pip,d_\pip,\mu_\pip)$ such that $u=1$ on $B_{d_\pip}(\zeta,r)\cap\partial\Om$, $0\le u\le 1$ on
		$\Om_\pip$, and $u=0$ on $\Om_\pip\setminus B_{d_\pip}(\zeta,2r)$. Then, 
		\begin{align*}
			\nu(B_{d_\pip}(\zeta,r)\cap\partial\Om)^{1/q}&\le \left(\int_{B_{d_\pip}(\zeta,r)}u^q\, d\nu\right)^{1/q}\\
			&\le \left(\int_{B_{d_\pip}(\zeta,4r)}u^q\, d\nu\right)^{1/q}\\
			&\lesssim \left(\int_{B_{d_\pip}(\zeta,4r)}|u-u_{B_{d_\pip}(\zeta,4r)}|^q\, d\mu\right)^{1/q}\\
			&\lesssim r^{1-\theta/q}\mu_\pip(B_{d_\pip}(\zeta,4r))^{1/q-1/p}\left(\int_{B_{d_\pip}(\zeta,2r)}g_{u,\pip}^p\, d\mu_\pip\right)^{1/p}.
		\end{align*}
		In obtaining the penultimate inequality, we used H\"older's inequality and the reverse doubling property of $\mu_\pip$, while in the ultimate inequality we used Theorem~\ref{thm:Adams} applied to $B_{d_\pip}(\zeta,4r)$ and the doubling property of $\mu_\pip$.
		Now an application of the  
		codimensionality relationship between $\nu$ and $\mu$ together with the fact that for
		$r<1/2$ the measure $\mu_\pip=\mu$, and then taking the infimum over all such $u$ on the right-hand side, yields the desired
		inequality.
		
		To verify the second claim of the corollary, we use the results of~\cite{BMS}. Note that from \cite{BB}, for each 
		$\zeta\in\partial\Om$
		and $r>0$, we have
		\[
		\rcapa_p^\pip(B_{d_\pip}(\zeta, r), B_{d_\pip}(\zeta, 2r))\approx \frac{\mu_\pip(B_{d_\pip}(\zeta,r))}{r^p}.
		\]
		So if $0<r<1/2$, then using the codimensionality of $\nu$ with respect to $\mu$ (and hence $\mu_\pip$), we see that
		\begin{align*}
			\rcapa_p^\pip(B_{d_\pip}(\zeta, r), B_{d_\pip}(\zeta, 2r))&\approx \frac{\nu(B_{d_\pip}(\zeta, r)\cap\partial\Om)}{r^{p-\theta}}\\
			& \lesssim \rcapa_p^\pip(B_{d_\pip}(\zeta,r)\cap\partial\Om,B_{d_\pip}(\zeta,2r)).
		\end{align*}
		Thus, $\partial\Om$ is uniformly $p$-fat with respect to the domain $\Om_\pip$, that is,
		\[
		\frac{\rcapa_p^\pip(B_{d_\pip}(\zeta,r)\cap\partial\Om,B_{d_\pip}(\zeta,2r))}
		{\rcapa_p^\pip(B_{d_\pip}(\zeta, r), B_{d_\pip}(\zeta, 2r))}\gtrsim 1,
		\]
		and so by the results in~\cite[Theorem~5.1]{BMS} the conclusion follows.
	\end{proof}

	\section{Effect on the Dirichlet problem for $(\Om, d,\mu)$}\label{Sec:nine}
	
	Now we have the tools necessary to verify Theorem~\ref{thm:one-point-two}. We split the proof into two theorems below.
	
	To show existence of solutions to a Dirichlet problem for bounded domains can be done via the direct method
	of calculus of variations: any sequence of functions, with the same boundary value, and with $p$-energy converging
	to the infimum of the $p$-energies of the class of all Dirichlet-Sobolev functions satisfying the fixed boundary conditions
	is done by first showing that this sequence is bounded in the Sobolev class (which is reflexive), and this in turn is 
	accomplished by using the Poincar\'e inequality, as the domain is bounded and hence sits inside a large ball
	to which the Poincar\'e inequality can be applied. 
	See~\cite{S2} for details regarding this method in the setting of metric spaces.
	When the domain is unbounded however, this method cannot work.
	
	In this final section of the paper, 
	we use the tools developed in the previous two sections to study existence and uniqueness issues related
	to solutions to the Dirichlet problem regarding $p$-harmonic functions on an unbounded domain whose boundary
	is bounded. In~\cite{Hansevi} a Perron method from~\cite{BBS-Per} was adapted to solve the Dirichlet problem
	corresponding to continuous boundary data on unbounded domains that were $p$-parabolic. As in~\cite{HoSh},
	an unbounded domain
	$\Om$ is \emph{$p$-parabolic} if $\Mod_p(\Gamma_\infty)=0$, where $\Gamma_\infty$ 
	consists of all locally rectifiable curves in
	$\Om$ that leave every compact subset of $\overline{\Om}^d$. Note that in our setting, curves in
	$\Gamma_\infty$ are the restriction to $\Om$ of the curves in $\Gamma$ studied in Proposition~\ref{thm:bdry-mod},
	and so by~\eqref{eq:chain-rule}, the domain $\Om$ is $p$-parabolic in the sense of~\cite{Hansevi} if and only if
	$\Mod_p^\pip(\Gamma)=0$. The domain is said to be \emph{$p$-hyperbolic} if it is not $p$-parabolic.
	
	The following theorem extends the result of~\cite{Hansevi} to boundary data in Besov classes.
	
	\begin{theorem}\label{thm:parabolicDirich}
		Let $1< p<\infty$, $(\Om,d,\mu)$ be a doubling metric measure space satisfying a $p$-Poincar\'{e} inequality such that $(\Om,d)$ is a locally compact, non-complete, unbounded uniform domain with bounded boundary, and $\nu$ a Borel regular measure that is $\theta$-codimensional with respect to $\mu$ for some $0<\theta<p$.  
		
		Suppose that $\Om$ is $p$-parabolic. Then, for every $f\in B^{1-\theta/p}_{p,p}(\partial\Om,\nu)$, there is a unique
		function $u\in D^{1,p}(\Om,\mu)$ such that
		\begin{itemize}
			\item $u$ is $p$-harmonic in $(\Om,d,\mu)$,
			\item $Tu=f$ on $\partial\Om$ $\nu$-a.e..
		\end{itemize}
	\end{theorem}
	
	\begin{proof}
		Recall from Proposition~\ref{lem:D-to-N} that 
		\[
		D^{1,p}(\Om,d,\mu)=N^{1,p}(\Om_\pip,d_\pip,\mu_\pip).
		\]
		Hence it suffices to find a function $u\in N^{1,p}(\Om_\pip,d_\pip,\mu_\pip)$ that is $p$-harmonic in
		$(\Om_\pip,d_\pip,\mu_\pip)$, for $p$-harmonicity in this metric measure space implies
		$p$-harmonicity in the original metric measure space $(\Om,d,\mu)$ thanks to 
		Proposition~\ref{prop:harm-harm}.
		
		Since $f\in B^{1-\theta/p}_{p,p}(\partial\Om,\nu)$, and this Besov space is the trace space of $N^{1,p}(\Om_\pip)$ by
		the proof of Proposition~\ref{prop:Dp-trace}, and as $\capa_p^\pip(\partial\Om)>0$ by Proposition~\ref{prop:bdryCap},
		the direct method of calculus of variations as described in the comment before the statement of the theorem applies; note
		that we also have $(\Om_\pip,d_\pip,\mu_\pip)$ is doubling and supports a $p$-Poincar\'e inequality. Thus
		there is a function $u\in N^{1,p}(\Om_\pip,d_\pip,\mu_\pip)$ that is $p$-harmonic in $\Om_\pip$ with respect to
		the metric $d_\pip$ and the measure $\mu_\pip$, such that $Tu=f$.
		
		Now suppose that $v$ is also a $p$-harmonic function in $(\Om,d,\mu)$ with $Tv=f$. We also have that $v$ is
		$p$-harmonic in $\Om_\pip\setminus\{\infty\}$ by Proposition~\ref{prop:harm-harm}. As $\capa_p^\pip(\{\infty\})=0$
		by the assumption of $p$-parabolicity, it follows from~\cite{Bjo} (see Remark~\ref{rem:7.8} above) that
		$v$ is $p$-harmonic in $\Om_\pip$. By the uniqueness of solutions to Dirichlet problem in $(\Om_\pip, d_\pip,\mu_\pip)$
		(see~\cite{S2}), we see that necessarily $v=u$ in $\Om_\pip$ and hence in $\Om$. This demonstrates the uniqueness of
		the solution.
	\end{proof}
	
	\begin{theorem}\label{thm:hyperbolicDirich}
		Let $1< p<\infty$, $(\Om,d,\mu)$ be a doubling metric measure space satisfying a $p$-Poincar\'{e} inequality such that $(\Om,d)$ is a locally compact, non-complete, unbounded uniform domain with bounded boundary, and $\nu$ a Borel regular measure that is $\theta$-codimensional with respect to $\mu$ for some $0<\theta<p$.  
		
		Suppose that $\Om$ is $p$-hyperbolic. Then, for every $f\in B^{1-\theta/p}_{p,p}(\partial\Om,\nu)$, there is a
		function $u\in D^{1,p}(\Om,\mu)$ such that
		\begin{itemize}
			\item $u$ is $p$-harmonic in $(\Om,d,\mu)$,
			\item $Tu=f$ on $\partial\Om$ $\nu$-a.e.,
			\item $\lim_{\Om\ni y\to\infty}u(y)$ exists as a real value. 
		\end{itemize}
	\end{theorem}
	
	In the above theorem, by $\lim_{\Om\ni y\to\infty}u(y)$ we mean a real number $\tau$ such that
	for some (and hence each) $x_0\in\Om$, for each $\eps>0$ we can find $R>0$ such that 
	$|u(y)-\tau|<\eps$ whenever $y\in\Om$ with $d(x_0,y)>R$.
	
	\begin{proof}
		As in the proof of Theorem~\ref{thm:parabolicDirich}, we see that solutions do exist, but unlike in that theorem, we do
		not have uniqueness. However, as $\Om$ is $p$-hyperbolic, we have that $\capa_p^\pip(\{\infty\})>0$.
		Let $v$ be any such $p$-harmonic solution; then $v\in N^{1,p}(\Om_\pip,d_\pip,\mu_\pip)$ by 
		Proposition~\ref{lem:D-to-N}.
		Thus, extending $f\in B^{1-\theta/p}_{p,p}(\partial\Om,\nu)$ to $\infty$ by setting
		\[
		f(\infty):=\lim_{r\to 0^+}\vint_{B_{d_\pip}(\infty,r)}v\, d\mu_\pip,
		\]
		and noting that $p$-capacity almost every point is a Lebesgue point of $v$ (see~\cite[Theorem~9.2.8]{HKST}), the above
		extension is well-defined. Thus $v$ solves the Dirichlet problem on $(\Om_\pip\setminus\{\infty\}, d_\pip,\mu_\pip)$
		with boundary data as the extended function $f$, with $f$ continuous at (and hence, as $\infty$ is an isolated
		point of $\partial[\Om_\pip\setminus\{\infty\}]$, in a neighborhood of) $\infty$. By Corollary~\ref{cor:relcap-lowerbd},
		we know that $\infty$ is a regular point for the domain $(\Om_\pip\setminus\{\infty\}, d_\pip,\mu_\pip)$.
		Hence $\lim_{\Om\ni y\to\infty}v(y)=f(\infty)$ exists.
	\end{proof}
	
	We in fact obtain more; if $v$ and $u$ are both solutions to the original problem, and
	if $\lim_{\Om\ni y\to\infty}v(y)=\lim_{\Om\ni y\to\infty}u(y)$, then necessarily $v=u$ in $\Om$ by the
	uniqueness of solutions to Dirichlet problem in $(\Om_\pip\setminus\{\infty\}, d_\pip,\mu_\pip)$ with 
	boundary data on $\partial[\Om_\pip\setminus\{\infty\}]=\partial\Om\cup\{\infty\}$. Moreover, of all the solutions
	to the Dirichlet problem on $\Om$ with boundary data $f$, there is only one solution that is also a solution in
	the domain $\Om_\pip$ with respect to the metric $d_\pip$ and the measure $\mu_\pip$.

	\section{Some illustrative examples}
	
	In this section, we consider some examples.
	
	\begin{example}
		Let $Z=[-1,1]\times[0,\infty)$, equipped with the restriction of the Euclidean metric and the $2$-dimensional Lebesgue
		measure from $\R^2$, and $\Om=[-1,1]\times(0,\infty)$. Then $\partial\Om=[-1,1]\times\{0\}$ is bounded. We have that
		the measure on $Z$ is doubling and supports the strongest of all Poincar\'e inequality, the $1$-Poincar\'e inequality. Moreover,
		$\Om$ is a uniform domain that is also unbounded. For $\beta>1$, with the choice of $\pip(t)=\min\{1,t^{-\beta}\}$ for
		$t>0$, the domain $\Om$ is transformed into $\Om_\pip=\Om\cup\{\infty\}$. Note that in the proof of 
		Proposition~\ref{thm:bdry-mod}, we need only consider the indices $Q_\mu^-$ and $Q_\mu^+$ corresponding to large values of
		$r$ and $R$; and in this case, we can set $Q_\mu^-=Q_\mu^+=1$, though when considering all scales of $r$,
		we have $Q_\mu^-=2$ and
		$Q_\mu^+=1$. So reading the hypotheses of Proposition~\ref{thm:bdry-mod} in this setting, when $1=Q_\mu^{\pm}<p$,
		we are in the case~(2) of the proposition, and then $\Mod_p^\pip(\Gamma)=0$. It follows that $\Om$ is $p$-parabolic,
		and the solution to the Dirichlet problem on $\Om$ is unique. Note that $\nu=\mathcal{H}^1\vert_{[-1,1]\times\{0\}}$
		has co-dimensional relation with respect to the $2$-dimensional Lebesgue measure at scales $0<r\le R_0$ with
		$\theta=1$. Solutions to the Dirichlet problem satisfy a homogeneous Neumann condition on the half-lines $\{\pm 1\}\times(0,\infty)$
		when seen as a function on $(-1,1)\times(0,\infty)$. For $p>1=\theta$, the Dirichlet problem is always solvable,
		and the solution is unique.
	\end{example}
	
	\begin{example}
		$Z:=\{(x,y)\in\R^2:y\ge\max\{0,|x|-1\} \}$ again be equipped the restriction of the Euclidean metric and Lebesgue 
		measure from $\R^2$, and $\Om=Z\setminus[-1,1]\times\{0\}$.
		In this case, at all scales of $r$ and $R$, we have $Q_\mu^-=Q_\mu^+=2$, and $\Om$ is $p$-parabolic when $p\ge 2$
		and is $p$-hyperbolic when $1\le p<2$. Moreover, we again have $\theta=1$. Here solutions to the Dirichlet
		problem on $\Om$ satisfy a homogeneous Neumann condition on the rays $\{(x,x-1)\, :\, x>1\}$ and
		$\{(x,1-x)\, :\, x<-1\}$. Here, for $p>1$ the Dirichlet problem is always solvable, but the solution is unique only when
		$p\ge 2$.
		
		On the other hand, if, with $Z$ as above, we have $\Om=Z\setminus K\times\{0\}$ with $K$ the standard $1/3$-rd Cantor set,
		then again $\Om$ is a uniform domain that is $p$-parabolic when $p\ge 2$ and $p$-hyperbolic when $1\le p<2$,
		but now 
		with $\nu$ the $\tfrac{\log 2}{\log 3}$-dimensional Hausdorff measure supported on $K\times\{0\}=\partial\Om$, we have that
		$\theta=2-\tfrac{\log 2}{\log 3}>1$. In this case, the Dirichlet problem is always solvable when $p>2-\tfrac{\log 2}{\log 3}$,
		but the solution is unique only when $p\ge 2$.
		Solutions to the Dirichlet problem on $\Om$ satisfy the homogeneous
		Neumann condition on the two above-mentioned rays, but in addition, they also satisfy that Neumann condition on
		$([-1,1]\setminus K)\times\{0\}$.
	\end{example}
	
	In contrast to the above examples, the domain $\Om=\{(x,y)\in \R^2\, :\, |y|>\max\{0,|x|-1\}\}\cup (-1,1)\times\{0\}$,
	considered as a domain in $Z=\R^2$, is not a uniform domain and hence the mechanisms developed in this paper do
	not apply to such $\Om$. However, from the fact that this domain is obtained by gluing two copies of the second example
	domain along their common boundary, we see that $\Om$ is $p$-hyperbolic precisely when $1\le p<2$, and this
	example indicates that $\overline{\Om}$ should have a two-point compactification rather than the one-point compactification
	considered here. We will not address this issue further in the present paper.
	
	\begin{example}
		With $Z=\R^2$ equipped with the $2$-dimensional Lebesgue measure and the Euclidean metric, let 
		$\Om=Z\setminus (K\times K)$, where $K$ is the standard $1/3$-rd Cantor set. Then $\Om$ is a uniform domain, and it is $p$-hyperbolic precisely when
		$1\le p<2$. In this case, $\nu$ is the restriction, to $\partial\Om=K\times K$, of the $\tfrac{\log 2}{\log 3}$-dimensional
		Hausdorff measure; thus $\theta=2[1-\tfrac{\log 2}{\log 3}]<1$. In this case the Dirichlet problem is solvable for
		each $p>1$, but the solution is unique only for $p\ge 2$.
	\end{example}
	
	\section*{Conflict of interest and data availability}
	
	On behalf of all authors, the corresponding author states that there is no conflict of interest. Data sharing 
	is not applicable to this article as no datasets were generated or analyzed during the current study.

	\noindent Address:\\
	
	\noindent R.G.: Department of Mathematical Sciences, P.O. Box 210025, University of
	Cincinnati, Cincinnati, OH 45221--0025, U.S.A. \\
	\noindent E-mail: {\tt ryan.gibara@gmail.com}\\
	
	\noindent R.K.: Aalto University
	Department of Mathematics and Systems Analysis, P.O.~Box~11100,  FI-00076 Aalto,  Finland. \\
	\noindent E-mail: {\tt riikka.korte@aalto.fi}\\
	
	\noindent N.S.: Department of Mathematical Sciences, P.O. Box 210025, University of
	Cincinnati, Cincinnati, OH 45221--0025, U.S.A. \\
	\noindent E-mail: {\tt shanmun@uc.edu} 
	
\end{document}